\documentclass{amsart}%
\usepackage{amsfonts}
\usepackage{amsmath}
\usepackage{amssymb}
\usepackage{graphicx}%
\newtheorem{theorem}{Theorem}
\theoremstyle{plain}

\newtheorem{claim}{Claim}

\newtheorem{conjecture}{Conjecture}

\newtheorem{definition}{Definition}

\newtheorem{lemma}{Lemma}

\newtheorem{problem}{Problem}
\newtheorem{proposition}{Proposition}
\newtheorem{remark}{Remark}

\numberwithin{equation}{section}
\begin{document}
\title[ ]{Local Zeta Functions Supported on Analytic Submanifolds and Newton Polyhedra}
\author{W. A. Zu\~{n}iga-Galindo}
\address{Centro de Investigacion y de Estudios Avanzados del I.P.N., Departamento de
Matematicas, Av. Instituto Politecnico Nacional 2508, Col. San Pedro
Zacatenco, Mexico D.F., C.P. 07360, Mexico}
\email{wzuniga@math.cinvestav.edu}

\begin{abstract}
The local zeta functions (also called Igusa's zeta functions) over $p$-adic
fields are connected with the number of solutions of congruences and
exponential sums mod $p^{m}$. These zeta functions are defined as integrals
over open and compact subsets with respect to the Haar measure. In this paper,
we introduce new integrals defined over submanifolds, \ or \ \ more
\ \ generally, \ \ over\ \ \ non-degenerate \ \ complete \ intersection
\ varieties, and study their connections with some arithmetical \ problems
such as estimation of \ exponential \ sums \ mod $p^{m}$. In particular we
\ extend \ Igusa's \ method \ for estimating \ exponential \ sums mod $p^{m}$
to the case of \ exponential \ sums \ mod $p^{m}$ along non-degenerate smooth varieties.

\end{abstract}
\maketitle

\section{Introduction}

Let $K$ be a $p-$adic field, i.e. $[K:\mathbb{Q}_{p}]<\infty$. Let $R_{K}$\ be
the valuation ring of $K$, $P_{K}$ the maximal ideal of $R_{K}$, and
$\overline{K}=R_{K}/P_{K}$ \ the residue field of $K$. The cardinality of the
residue field of $K$ is denoted by $q$, thus $\overline{K}=\mathbb{F}_{q}$.
For $z\in K$, $ord\left(  z\right)  \in\mathbb{Z}\cup\{+\infty\}$ \ denotes
the valuation of $z$, and $\left\vert z\right\vert _{K}=q^{-ord\left(
z\right)  }$, $ac$ $z=z\pi^{-ord(z)}$, where $\pi$ is a fixed uniformizing
parameter of $R_{K}$.

Let $f_{1},\ldots,f_{l}$ be polynomials in $K\left[  x_{1},\ldots
,x_{n}\right]  $, or, more generally, $K-$ana\-ly\-tic functions on an open
and compact set $U\subset K^{n}$. For $1\leq j\leq l$ we define the
$K-$analytic set $V^{(j)}(K):=V^{(j)}=\left\{  x\in U\mid f_{i}\left(
x\right)  =0,\text{ }1\leq i\leq j\right\}  $. We assume \ that $V^{(l-1)}$ is
a \ closed submanifold of $U$, i.e. $rank_{K}\left(  \frac{\partial f_{i}%
}{\partial x_{j}}(z)\right)  =l-1$ for any $z\in V^{(l-1)}$, and that
$f_{l}:V^{(l-1)}\rightarrow K$ is an analytic function on $V^{(l-1)}$. Let
$\Phi:K^{n}\rightarrow\mathbb{C}$ \ be a Bruhat-Schwartz function (with
support in $U$ in the second case). Let $\omega$ be a quasicharacter of
$K^{\times}$. To these data we associate the following local zeta function:
\begin{align*}
Z_{\Phi}(\omega,f_{1},\ldots,f_{l},K)  &  :=Z_{\Phi}(\omega,V^{(l-1)},f_{l})\\
&  =\int\limits_{V^{(l-1)}(K)}\Phi\left(  x\right)  \omega\left(  \text{
}f_{l}(x)\right)  \mid\gamma_{GL}\left(  x\right)  \mid,
\end{align*}
for $\omega\in\Omega_{0}\left(  K^{\times}\right)  $, where $\mid\gamma
_{GL}\left(  x\right)  \mid$ is the measure induced by a Gel'fand-Leray form
on $V^{(l-1)}(K)$ (see Section 2). In this paper we provide a geometric
description of the poles of the meromorphic continuation of $Z_{\Phi}%
(\omega,f_{1},\ldots,f_{l},K)$ when $f_{1},\ldots,f_{l}$ are non-degenerate
with respect to their Newton polyhedra (see Theorem \ref{theorem1} and Remark
\ref{remark5}). The relevance of this problem is well-understood in the case
in which $V^{(l-1)}$ is an open subset of $K^{n}$ (see e.g. \cite{D0},
\cite{D1a}, \cite{D-H}, \cite{L-M}, \cite{M1}, \cite{S-Z}, \cite{V1a},
\cite{V2}, \cite{V-Z}, \cite{Z0}, \cite{Z1}, \cite{Z2}).

The main motivation for this paper is the estimation of exponential sums of
the type
\[
E(z):=q^{-m(n-l+1)}%
{\textstyle\sum\nolimits_{x\in V^{(l-1)}\left(  R_{K}/P_{K}^{m}\right)  }}
\Psi\left(  zf_{j+1}(x)\right)  ,
\]
where $\left\vert z\right\vert _{K}=q^{m\text{ }}$ with $m\in\mathbb{N}$,
$\Psi\left(  \cdot\right)  $ is the standard additive character of $K$, and
$V^{(l-1)}$ is a smooth algebraic variety defined over $R_{K}$. In \cite{Ka},
(see also \cite{D-F}, \cite{Z1A}), Katz gave a stationary phase formula for
$E(z)$ when $f_{l}$ has Morse singularities $\operatorname{mod}$ $P_{K}$. In
\cite{Mor}\ Moreno proposed to extend Igusa's method (see e.g. \cite{D0},
\cite{I2}) to exponential sums of type $E(z)$.

A more \ general problem is to estimate oscillatory integrals of type%

\[
E_{\Phi}(z,V^{(l-1)},f_{l},K):=E_{\Phi}(z,V^{(l-1)},f_{l})=\int
\limits_{V^{(l-1)}(K)}\Phi\left(  x\right)  \Psi\left(  zf_{l}(x)\right)
\mid\gamma_{GL}\left(  x\right)  \mid,
\]
for $\left\vert z\right\vert _{K}\gg0$. Indeed, if $V^{(l-1)}\left(  K\right)
$ has good reduction $\operatorname{mod}$ $P_{K}$, then $E(z)$ can be
expressed as an integral of the previous type (cf. Lemma \ref{lemma4}). The
relevance of studying integrals of type $E_{\Phi}(z,V^{(l-1)},f_{l})$ was
pointed out in \cite{Kazh} by Kazhdan. In this paper we extend Igusa's method
to oscillatory integrals of type $E_{\Phi}(z,V^{(l-1)},f_{l})$ when $f_{i}$,
$i=1,\ldots,l$ are non-degenerate with respect to their Newton polyhedra; more
precisely, we show the existence of an asymptotic expansion for $E_{\Phi
}(z,V^{(l-1)},f_{l})$, $\left\vert z\right\vert _{K}\gg0$, which is controlled
by the poles of $Z_{\Phi}(\omega,V^{(l-1)},f_{l})$ (see Theorems
\ref{theorem2}, \ref{theorem3}, Remark \ref{remark10}, and Theorem
\ref{theorem6}). We are also able to estimate the number of solutions of a
polynomial congruence over a smooth algebraic variety (see Theorem
\ref{theorem5}).

At this point, it is worth to mention \ that there are other zeta functions
supported on analytic sets. The na\"{\i}ve local zeta function supported on a
submanifold is defined as
\[
Z_{\Phi}^{\text{na\"{\i}ve}}\left(  \omega,V^{(l-1)},f_{l}\right)
=\int\limits_{K^{n}}\Phi\left(  x\right)  \delta\left(  f_{1}(x),\ldots
,f_{l-1}(x)\right)  \omega\left(  \text{ }f_{l}(x)\right)  \mid dx\mid,
\]
for $\omega\in\Omega_{0}\left(  K^{\times}\right)  $, where $\delta\left(
\cdot\right)  $ is the Dirac delta function and $\mid dx\mid$ is the Haar
measure of $K^{n}$ normalized so that volume of $R_{K}^{n}$ is one. The
definition of `integrals' of type $Z_{\Phi}^{\text{na\"{\i}ve}}\left(
\omega,V^{(l-1)},f_{l}\right)  $\ is based on the fact that an expression of
form $\int_{K^{n}}\Phi\left(  x\right)  \delta\left(  f_{1}(x),\ldots
,f_{l-1}(x)\right)  \mid dx\mid$ gives a well-defined linear functional on the
Bruhat-Schwartz space if $V^{(l-1)}$\ is a submanifold (see \cite{G-S} and
Section \ref{DeltaSect}). Furthermore, $Z_{\Phi}^{\text{na\"{\i}ve}}\left(
\omega,V^{(l-1)},f_{l}\right)  =Z_{\Phi}(\omega,V^{(l-1)},f_{l})$ (see Lemma
\ref{lemma1}).

In \cite{Has}\ Hashimoto studied local zeta functions on $\mathbb{R}^{n}$ and
$l=2$ supported on analytic sets, in the case in which $V^{(1)}\left(
\mathbb{R}\right)  $ is a submanifold. Hashimoto showed the existence of an
asymptotic expansion for an oscillanting integral supported on $V^{(1)}\left(
\mathbb{R}\right)  $ which is controlled by the poles of local zeta functions
(\cite[Theorem 14]{Has}). This result is an extension of Varchenko's result on
oscillanting integrals and Newton \ polyhedra \cite{Var}. Hashimoto asserts
that a similar result holds when $V^{(1)}\left(  \mathbb{R}\right)  $\ is
singular, more precisely, if $V^{(1)}\left(  \mathbb{R}\right)  $ \ is a
non-degenerate \ complete intersection singularity at the origin
(\cite[Theorem 27]{Has}). However, Hashimoto does not prove that the
na\"{\i}ve \ local zeta functions are `true integrals' on some half-plane of
the complex plane. To the best knowledge of the author, this is a crucial
point to establish an asymptotic expansion for oscillanting integrals
supported on an analytic subset. In the real and $p$-adic cases, when
$V^{(l-1)}$ is singular, `integrals' of type $Z_{\Phi}^{\text{na\"{\i}ve}%
}\left(  \omega,V^{(l-1)},f_{l}\right)  $ are just `symbols'. However, under
suitable hypotheses on $V^{(l-1)}$\ and by using a toroidal resolution of
singularities, it is possible to associate to a `symbol' of type $Z_{\Phi
}^{\text{na\"{\i}ve}}\left(  \omega,V^{(l-1)},f_{l}\right)  $ a meromorphic
function, which depends on the choice of the resolution.

In the $p$-adic case, to circumvent the above mentioned problem, we introduce
the following local zeta function:%

\begin{align*}
\mathcal{Z}_{\Phi}\left(  s,V^{(l-1)},f_{l}\right)   &  :=\lim_{r\rightarrow
+\infty}%
{\displaystyle\int\limits_{K^{n}}}
\Phi\left(  x\right)  \delta_{r}\left(  f_{1}(x),\ldots,f_{l-1}(x)\right)
\left\vert f_{l}\left(  x\right)  \right\vert _{K}^{s}\mid dx\mid\\
&  =\lim_{r\rightarrow+\infty}%
{\displaystyle\int\limits_{S_{r}}}
\Phi\left(  x\right)  q^{r\left(  l-1\right)  }\left\vert f_{l}\left(
x\right)  \right\vert _{K}^{s}\mid dx\mid,
\end{align*}
where $\delta_{r}$ is a sequence of functions satisfying $\lim_{r\rightarrow
+\infty}\delta_{r}=\delta$, $s\in\mathbb{C}$ with $\operatorname{Re}(s)>0$,
and $S_{r}:=\left\{  x\in K^{n}\mid ord(f_{i}(x))\geqslant r\text{,
}i=1,\ldots,l-1\right\}  $. The integrals of type $\mathcal{Z}_{\Phi}\left(
s,V^{(l-1)},f_{l}\right)  $ are limits of the integrals considered by Denef in
\cite{D1}.

In the case in which $f_{i}$, $i=1,\ldots,l$ are polynomials defined over
$R_{K}$ which are non-degenerate with respect to their Newton polyhedra mod
$P_{K}$, we show \ that $\mathcal{Z}_{\Phi}\left(  s,V^{(l-1)},f_{l}\right)  $
\ defines a holomorphic function for $\operatorname{Re}(s)$ sufficiently big,
in addition, it has a meromorphic continuation to the complex plane as a
rational function of $q^{-s}$ which can be computed in terms of the Newton
polyhedra of the $f_{i}$, see Theorem \ref{theorem7}. The zeta functions
$\mathcal{Z}_{\Phi}\left(  s,V^{(l-1)},f_{l}\right)  $\ may have poles with
positive real parts (see Example \ref{Example2}), and several examples
\ suggest that the real parts of the poles are eigenvalues of the $l$-th
principal monodromy introduced by Oka (see Example \ref{example1}, Conjecture
\ref{conj}, and references \cite{O1}-\cite{O2}).

Finally, we want to mention that there are several open questions, problems
and conjectures connected with the local zeta functions introduced in this
paper. We have formulated some of them along this paper.

\textbf{Acknowledgement.} The author thanks to the anonymous referees for
their careful reading of the original version this of paper, and for several
useful comments and notes that helped him to improve the original version.

\textbf{Remark} The version to be pusblished in IMRN contains two figures.

\section{Local Zeta Functions Supported on Analytic Submanifolds}

Let $f_{i}$ be a polynomial in $K\left[  x_{1},\ldots,x_{n}\right]  $,
$f_{i}\left(  0\right)  =0$, or, more generally, a $K-$ana\-ly\-tic function
on an open and compact set $U\subset K^{n}$, for $i=1,\ldots,l$. We assume
that $\ 2\leq l\leq n$ and set%
\[
V^{(j)}(K)=V^{(j)}=\left\{  x\in U\mid f_{i}\left(  x\right)  =0,\text{ }1\leq
i\leq j\right\}  ,
\]
as in the introduction. From now on, we will assume that $V^{(j)}(K)$ is a
closed submanifold of $U$ of dimension $n-j$. We refer the reader to
\cite{I1}, \cite{Ser}\ for general results on $K$-analytic manifolds.

\begin{remark}
\label{remark-1}All the $K$-analytic functions on an open set $U$ which are
considered in this paper are given by one power series which is convergent on
the whole set $U$. Since $U$ is totally disconnected, we can take a finite
number of open subsets $U_{i}\subset U$, which form a partition of $U$, and if
we define on each $U_{i}$ an analytic function by using a power series $f_{i}%
$, we obtain an analytic function on $U$.
\end{remark}

Let $\Phi:K^{n}\rightarrow\mathbb{C}$ \ be a Bruhat-Schwartz function (with
support in $U$ in the second case). Let $S(K^{n})$ be the $\mathbb{C}$-vector
space of Bruhat-Schwartz functions.

Let $\omega$ be a quasicharacter of $K^{\times}$, i.e., a continuous
homomorphism from $K^{\times}$ into $\mathbb{C}^{\times}$. The set of
quasicharacters form an Abelian group denoted as $\Omega\left(  K^{\times
}\right)  $. We define an element $\omega_{s}$ of $\Omega\left(  K^{\times
}\right)  $ for every $s\in\mathbb{C}$ as $\omega_{s}\left(  x\right)
=\left\vert x\right\vert _{K}^{s}=q^{-sord\left(  x\right)  }$. If, for every
$\omega$ in $\Omega\left(  K^{\times}\right)  $, we choose $s\in\mathbb{C}%
$\ satisfying $\omega\left(  \pi\right)  =q^{-s}$, then $\omega\left(
x\right)  =\omega_{s}\left(  x\right)  \chi\left(  ac\text{ }x\right)  $ in
which $\chi:=\omega\mid_{R_{K}^{\times}}$. Hence $\Omega\left(  K^{\times
}\right)  $ is isomorphic to $\mathbb{C\times}\left(  R_{K}^{\times}\right)
^{\ast}$, where $\left(  R_{K}^{\times}\right)  ^{\ast}$ is the group of
characters of $R_{K}^{\times}$, and $\Omega\left(  K^{\times}\right)  $ is a
one dimensional complex manifold. We note that $\sigma\left(  \omega\right)
:=\operatorname{Re}(s)$ depends only on $\omega$, and $\left\vert
\omega\left(  x\right)  \right\vert =\omega_{\sigma\left(  \omega\right)
}\left(  x\right)  $. We define an open subset of $\Omega\left(  K^{\times
}\right)  $ by
\[
\Omega_{\sigma}\left(  K^{\times}\right)  =\left\{  \omega\in\Omega\left(
K^{\times}\right)  \mid\sigma\left(  \omega\right)  >\sigma\right\}  .
\]
For further details we refer the reader to \cite{I2}. To above data we
associate the following local zeta function:
\begin{align*}
Z_{\Phi}(\omega,f_{1},\ldots,f_{l},K)  &  :=Z_{\Phi}(\omega,V^{(l-1)},f_{l})\\
&  =\int\limits_{V^{(l-1)}\left(  K\right)  }\Phi\left(  x\right)
\omega\left(  f_{l}(x)\right)  \mid\gamma_{GL}\left(  x\right)  \mid,
\end{align*}
for $\omega\in\Omega_{0}\left(  K^{\times}\right)  $, where $\mid\gamma
_{GL}\left(  x\right)  \mid$ is the measure induced on $V^{(l-1)}(K)$ by a
Gel'fand-Leray differential form, i.e., by a form satisfying $\gamma
_{GL}\wedge\wedge_{i=1}^{l-1}df_{i}=\wedge_{i=1}^{n}dx_{i}$. The
Gel'fand-Leray form is not unique, but $Z_{\Phi}(\omega,V^{(l-1)},f_{l})$ is
well-defined because the restriction of $\gamma_{GL}$ to $V^{(l-1)}$ is
independent of the choice of $\gamma_{GL}$, see \cite[Chap. III, Sect.
1-9]{G-S}. We warn the reader that $Z_{\Phi}(\omega,V^{(l-1)},f_{l})$ depends
on $f_{1},\ldots,f_{l}$ and not only on $V^{(l-1)}$ and $f_{l}$.

\begin{proposition}
\label{proposition0}The zeta function $Z_{\Phi}(\omega,V^{(l-1)},f_{l})$ is
holomorphic on $\Omega_{0}\left(  K^{\times}\right)  $.In addition, it has a
meromorphic continuation to the whole $\Omega\left(  K^{\times}\right)  $ as a
rational function of $t=\omega\left(  \pi\right)  $. The real parts of the
poles of the meromorphic continuation are negative rational numbers.
\end{proposition}

\begin{proof}
Given a point $b\in V^{(l-1)}\left(  K\right)  $, there exists an open compact
subset $W\subseteq K^{n}$ containing $b$, and a coordinate system $\phi\left(
x\right)  =\left(  y_{1},\ldots,y_{n}\right)  $, possible after renaming the
$x_{i}$'s, on $W$ such that $y_{i}=f_{i}\left(  x\right)  $, $i=1,\ldots,l-1$,
then
\[
V^{(l-1)}\left(  K\right)  =\left\{  y_{i}=0\text{, }i=1,\ldots,l-1\right\}
\]
locally, and $\wedge_{i=1}^{n}dy_{i}=J(x)\wedge_{i=1}^{n}dx_{i}$ on $W$, here
$J(x)$ is the Jacobian of $\phi\left(  x\right)  $, by schrinking $W$, if
necessary, we may assume that $\left\vert J(x)\right\vert _{K}=\left\vert
J(b)\right\vert _{K}$ for every $x\in W$.$\ $By passing to a sufficiently fine
disjoint covering of the support of $\Phi$, the zeta function $Z_{\Phi}%
(\omega,V^{(l-1)},f_{l})$ becomes a finite sum of Igusa's local zeta
functions,\ more precisely, a finite sum of integrals of type%
\begin{equation}
I\left(  \omega\right)  :=\left\vert J(b)\right\vert _{K}^{-1}%
{\displaystyle\int\limits_{K^{n-l+1}}}
\Theta\left(  y\right)  \omega\left(  h\left(  y\right)  \right)  \left\vert
{\textstyle\bigwedge\limits_{i=l}^{n}}
dy_{i}\right\vert , \label{Igusa}%
\end{equation}
where $\Theta$\ is a Bruhat-Schwartz function, and $h\left(  y\right)
:=\left(  f_{l}\circ\phi^{-1}\right)  \left(  0,\ldots,0,y_{l},\ldots
,y_{n}\right)  $ is a $K$-analytic function defined on an open subset
containing the support of $\Theta$. Now the results announced follow from the
corresponding results about Igusa's zeta function (see \cite[Theorem
8.2.1]{I2}).
\end{proof}

In this paper we will study the following problem in a toric setting:

\begin{problem}
Provide a geometric description of the poles of the meromorphic continuation
of $Z_{\Phi}(\omega,V^{(l-1)},f_{l})$ in terms of an embedded resolution of
singularities of the divisor $D_{K}:=\cup_{i=1}^{l}f_{i}^{-1}\left(  0\right)
$.
\end{problem}

\subsection{The Na\"{\i}ve Definition of $Z_{\Phi}(\omega,V^{(l-1)},f_{l})$}

\subsubsection{\label{DeltaSect}The Dirac Delta function}

Let $\delta$ denote the Dirac delta function:%
\[
\delta\left(  x\right)  =\left\{
\begin{array}
[c]{ccc}%
0 & \text{if} & x\neq0\\
&  & \\
+\infty & \text{if} & x=0,
\end{array}
\right.
\]
for $x\in K^{n}$. We set for $r\in\mathbb{N}$,
\[
\delta_{r}\left(  x\right)  =\left\{
\begin{array}
[c]{ccc}%
0 & \text{if} & x\notin\left(  \pi^{r}R_{K}\right)  ^{n}\\
&  & \\
q^{rn} & \text{if} & x\in\left(  \pi^{r}R_{K}\right)  ^{n}.
\end{array}
\right.
\]
Then $\lim_{r\rightarrow+\infty}\delta_{r}\left(  x\right)  =\delta\left(
x\right)  $, and $\int_{K^{n}}\delta_{r}\left(  x\right)  \left\vert
dx\right\vert =1$, for any $r\in\mathbb{N}$. We recall that if $\theta\in
S(K^{n})$, then
\[
\lim_{r\rightarrow+\infty}\int_{K^{n}}\theta\left(  x\right)  \delta
_{r}\left(  x\right)  \left\vert dx\right\vert =\int_{K^{n}}\theta\left(
x\right)  \delta\left(  x\right)  \left\vert dx\right\vert =\theta\left(
0\right)  ,
\]
i.e., $\lim_{r\rightarrow+\infty}\delta_{r}=\delta$ on $S(K^{n})$.

We now review the definition of `integrals' of type%
\begin{equation}
I_{\Phi}:=\int\limits_{K^{n}}\Phi\left(  x\right)  \delta\left(
f_{1}(x),\ldots,f_{l-1}(x)\right)  \mid dx\mid, \label{int}%
\end{equation}
following Gel'fand and Shilov's book \cite{G-S}, in the case in which $\Phi\in
S(K^{n})$ and $V^{(l-1)}\left(  K\right)  $\ is a submanifold. We may assume
without loss of generality that $\Phi$\ is the product of a constant $c$ by
the characteristic function of a ball $W:=b+\left(  \pi^{e_{0}}R_{K}\right)
^{n}$, with $b\in V^{(l-1)}\left(  K\right)  $. Since $V^{(l-1)}\left(
K\right)  $ is a submanifold, there exists a coordinate change of the form
$y=\left(  y_{1},\ldots,y_{n}\right)  =\phi\left(  x\right)  $, with
\[
y_{i}:=\left\{
\begin{array}
[c]{ccc}%
f_{i}\left(  x\right)  & \text{if} & i=1,\ldots,l-1\\
&  & \\
x_{i}-b_{i} & \text{if} & i=l,\ldots,n,
\end{array}
\right.
\]
such that $\phi:W\rightarrow\pi^{d_{1}}R_{K}\times\ldots\times\pi^{d_{n}}%
R_{K}$ is a $K$-analytic isomorphism, for some $\left(  d_{1},\ldots
,d_{n}\right)  \in\mathbb{N}^{n}$, whose Jacobian $J(x)$ satisfies $\left\vert
J(x)\right\vert _{K}=\left\vert J(b)\right\vert _{K}$, for any $x\in W$. By
using $y=\phi\left(  x\right)  $\ as a change of variables in (\ref{int}) we
define%
\begin{align*}
I_{\Phi}  &  =c\left\vert J(b)\right\vert _{K}^{-1}q^{-\sum_{j=l}^{n}d_{j}%
}\left(  \int\limits_{\pi^{d_{1}}R_{K}\times\ldots\times\pi^{d_{l-1}}R_{K}%
}\delta\left(  y_{1},\ldots,y_{l-1}\right)  \mid dy_{1}\ldots dy_{l-1}%
\mid\right) \\
&  =c\left\vert J(b)\right\vert _{K}^{-1}q^{-\sum_{j=l}^{n}d_{j}}.
\end{align*}
The previous definition is independent of the coordinate system used in the
calculation because
\[
I_{\Phi}=\int\limits_{V^{(l-1)}(K)}\Phi\left(  x\right)  \mid\gamma
_{GL}\left(  x\right)  \mid,
\]
where $\gamma_{GL}$ is a Gel'fand-Leray form on $V^{(l-1)}(K)$. Here we recall
that the restriction of $\gamma_{GL}$\ to $V^{(l-1)}(K)$ is unique, and then
the previous integral is well-defined \cite[Chap. III, Sect. 1-9]{G-S}.

\begin{remark}
\label{osci_int}Let $z\in K$. Since $\Phi\left(  x\right)  \Psi\left(
zf_{l}\left(  x\right)  \right)  $ is a Bruhat-Schwartz function, we can apply
this to obtain the following result for oscillatory integrals:%
\[%
{\textstyle\int\limits_{K^{n}}}
\Phi\left(  x\right)  \delta\left(  f_{1}(x),\ldots,f_{l-1}(x)\right)
\Psi\left(  zf_{l}\left(  x\right)  \right)  \mid dx\mid=\int
\limits_{V^{(l-1)}(K)}\Phi\left(  x\right)  \Psi\left(  zf_{l}\left(
x\right)  \right)  \mid\gamma_{GL}\left(  x\right)  \mid.
\]

\end{remark}

\subsubsection{The na\"{\i}ve definition}

The na\"{\i}ve local zeta function supported on a submanifold is defined as%
\[
Z_{\Phi}^{\text{na\"{\i}ve}}(\omega,V^{(l-1)},f_{l})=\int\limits_{K^{n}}%
\Phi\left(  x\right)  \delta\left(  f_{1}\left(  x\right)  ,\ldots
,f_{l-1}\left(  x\right)  \right)  \omega\left(  f_{l}(x)\right)  \mid
dx\mid,
\]
for $\omega\in\Omega_{0}\left(  K^{\times}\right)  $.

From the previous discussion about the integrals of type $I_{\Phi}$ and by
using the same reasoning as in the proof of Proposition \ref{proposition0} and
Remark \ref{remark-1} we obtain the following lemma.

\begin{lemma}
\label{lemma1}$Z_{\Phi}^{\text{na\"{\i}ve}}(\omega,V^{(l-1)},f_{l})=Z_{\Phi
}(\omega,V^{(l-1)},f_{l})$, for $\omega\in\Omega_{0}\left(  K^{\times}\right)
$.
\end{lemma}

Let $g_{i}:U\rightarrow K$, $g_{i}\left(  0\right)  =0$, $i=1,\ldots,l$,
$l\geq2$, be an analytic function on an open subset $U$. We set $W:=\left\{
x\in U\mid g_{i}\left(  x\right)  =0,i=1,\ldots,l-1\right\}  $. Another zeta
function is defined as follows:%

\begin{align*}
\mathcal{Z}_{\Phi}\left(  \omega,W,g_{l}\right)   &  :=\lim_{r\rightarrow
+\infty}%
{\displaystyle\int\limits_{K^{n}}}
\Phi\left(  x\right)  \delta_{r}\left(  g_{1}(x),\ldots,g_{l-1}(x)\right)
\omega\left(  g_{l}\left(  x\right)  \right)  \mid dx\mid\\
&  =\lim_{r\rightarrow+\infty}%
{\displaystyle\int\limits_{S_{r}}}
\Phi\left(  x\right)  q^{r\left(  l-1\right)  }\omega\left(  g_{l}\left(
x\right)  \right)  \mid dx\mid,
\end{align*}
where $S_{r}:=\left\{  x\in K^{n}\mid ord(g_{i}(x))\geqslant r\text{,
}i=1,\ldots,l-1\right\}  $. In Section \ref{section6} we will study these
integrals in a toric setting, in particular we will show the existence of the
limit. The $\mathcal{Z}_{\Phi}\left(  s,W,g_{l}\right)  $ are limits of the
integrals considered by Denef in \cite{D1}. We emphasize that $Z_{\Phi
}^{\text{na\"{\i}ve}}(\omega,W,g_{l})$ is not necessary equal to
$\mathcal{Z}_{\Phi}\left(  s,W,g_{l}\right)  $ because $\Phi\left(  x\right)
\omega\left(  g_{l}\left(  x\right)  \right)  $ is not a Bruhat-Schwartz function.

\section{Polyhedral Subdivisions of\textit{\ }$\mathbb{R}_{+}^{n}$ and
Resolution of Singularities}

\subsection{Newton polyhedra}

We set $\mathbb{R}_{+}:=\{x\in\mathbb{R}\mid x\geqslant0\}$. Let $G$ be a
non-empty subset of $\mathbb{N}^{n}$. The \textit{Newton polyhedron }%
$\Gamma=\Gamma\left(  G\right)  $ associated to $G$ is the convex hull in
$\mathbb{R}_{+}^{n}$ \ of the set $\cup_{m\in G}\left(  m+\mathbb{R}_{+}%
^{n}\right)  $. For instance classically one associates a \textit{Newton
polyhedron }$\Gamma\left(  g\right)  $\textit{\ (at the origin) to }
$g(x)=\sum_{m}c_{m}x^{m}$ ($x=\left(  x_{1},\ldots,x_{n}\right)  $, $g(0)=0$),
\ being a non-constant polynomial function over $K$ or $K-$analytic function
\ in a neighborhood of the origin, where $G=$supp$(g)$ $:=$ $\left\{
m\in\mathbb{N}^{n}\mid c_{m}\neq0\right\}  $.\ Further we will associate more
generally a Newton polyhedron to an analytic mapping.

We fix a Newton polyhedron $\Gamma$\ as above. We first collect some notions
and results about Newton polyhedra that \ will be used in the next sections.
Let $\left\langle \cdot,\cdot\right\rangle $ denote the usual inner product of
$\mathbb{R}^{n}$, and identify \ the dual space of $\mathbb{R}^{n}$ with
$\mathbb{R}^{n}$ itself by means of it.

For $a\in\mathbb{R}_{+}^{n}$, we define
\[
d(a,\Gamma)=\min_{x\in\Gamma}\left\langle a,x\right\rangle ,
\]
and \textit{the first meet locus }$F(a,\Gamma)$ of $a$ as \
\[
F(a,\Gamma):=\{x\in\Gamma\mid\left\langle a,x\right\rangle =d(a,\Gamma)\}.
\]
The first meet locus \ is a face of $\Gamma$. Moreover, if $a\neq0$,
$F(a,\Gamma)$ is a proper face of $\Gamma$.

If $\Gamma=\Gamma\left(  g\right)  $, we define the \textit{face function
}$g_{a}\left(  x\right)  $\textit{\ of }$g(x)$\textit{\ with respect to }$a$
as
\[
g_{a}\left(  x\right)  =g_{F(a,\Gamma)}\left(  x\right)  =\sum_{m\in
F(a,\Gamma)}c_{m}x^{m}.
\]

In the case of functions having subindices, say $g_{i}(x)$, we will use the
notation $g_{i,a}(x)$ for the face function of $g_{i}(x)$\ with respect to $a
$.

We will say that $a=\left(  a_{1},\ldots,a_{n}\right)  \in\mathbb{R}^{n}$ is
\textit{a positive vector} (respectively \textit{strictly positive vector}),
if $a_{i}\geq0$, for $i=1,\ldots,n$, (respectively if $a_{i}>0$, for
$i=1,\ldots,n$).\ We use the notation $a\succ0$ to mean that $a$ is strictly positive.

\subsection{Polyhedral Subdivisions Subordinate to a Polyhedron}

We define an equivalence relation in \ $\mathbb{R}_{+}^{n}$ by taking $a\sim
a^{\prime}\Leftrightarrow F(a,\Gamma)=F(a^{\prime},\Gamma)$. The equivalence
classes of \ $\sim$ are sets of the form
\[
\Delta_{\tau}=\{a\in\mathbb{R}_{+}^{n}\mid F(a,\Gamma)=\tau\},
\]
where $\tau$ \ is a face of $\Gamma$.

We recall that the cone strictly\ spanned \ by the vectors $a_{1},\ldots
,a_{r}\in\mathbb{R}_{+}^{n}\setminus\left\{  0\right\}  $ is the set
$\Delta=\left\{  \lambda_{1}a_{1}+...+\lambda_{r}a_{r}\mid\lambda_{i}%
\in\mathbb{R}_{+}\text{, }\lambda_{i}>0\right\}  $. If $a_{1},\ldots,a_{r}$
are linearly independent over $\mathbb{R}$, $\Delta$ \ is called \ a
\textit{simplicial cone}. \ If \ $a_{1},\ldots,a_{r}\in\mathbb{Z}^{n}$, we say
$\Delta$\ is \ a \textit{rational cone}. If $\left\{  a_{1},\ldots
,a_{r}\right\}  $ is a subset of a basis \ of the $\mathbb{Z}$-module
$\mathbb{Z}^{n}$, we call $\Delta$ a \textit{simple cone}.

A precise description of the geometry of the equivalence classes modulo $\sim$
is as follows. Each \textit{facet} (i.e. a face of codimension one)\ $\gamma$
of $\Gamma$\ has a unique vector $a(\gamma)=(a_{\gamma,1},\ldots,a_{\gamma
,n})\in\mathbb{N}^{n}\mathbb{\setminus}\left\{  0\right\}  $, \ whose nonzero
coordinates are relatively prime, which is perpendicular to $\gamma$. We
denote by $\mathfrak{D}(\Gamma)$ the set of such vectors. The equivalence
classes are rational cones of the form
\[
\Delta_{\tau}=\{\sum\limits_{i=1}^{r}\lambda_{i}a(\gamma_{i})\mid\lambda
_{i}\in\mathbb{R}_{+}\text{, }\lambda_{i}>0\},
\]
where $\tau$ runs through the set of faces of $\Gamma$, and $\gamma_{i}$,
$i=1,\ldots,r$\ are the facets containing $\tau$. We note that $\Delta_{\tau
}=\{0\}$ if and only if $\tau=\Gamma$. The family $\left\{  \Delta_{\tau
}\right\}  _{\tau}$, with $\tau$ running over\ the proper faces of $\Gamma$,
is a partition of $\mathbb{R}_{+}^{n}\backslash\{0\}$; we call this partition
a \textit{\ polyhedral subdivision of }\ $\mathbb{R}_{+}^{n}$%
\ \textit{subordinate} to $\Gamma$. We call $\left\{  \overline{\Delta}_{\tau
}\right\}  _{\tau}$, the family formed by the topological closures of the
$\Delta_{\tau}$, a \textit{\ fan} \textit{subordinate} to $\Gamma$.

Each cone $\Delta_{\tau}$\ can be partitioned \ into a finite number of
simplicial cones $\Delta_{\tau,i}$. In addition, the subdivision can be chosen
such that each $\Delta_{\tau,i}$ is spanned by part of $\mathfrak{D}(\Gamma)$.
Thus from the above considerations we have the following partition of
$\mathbb{R}_{+}^{n}\backslash\{0\}$:%
\[
\mathbb{R}_{+}^{n}\backslash\{0\}=\bigcup\limits_{\tau\text{ }}\left(
\bigcup\limits_{i=1}^{l_{\tau}}\Delta_{\tau,i}\right)  ,
\]
where $\tau$ runs \ over the proper faces of $\Gamma$, and each $\Delta
_{\tau,i}$ \ is a simplicial cone contained in $\Delta_{\tau}$.\ We will say
that $\left\{  \Delta_{\tau,i}\right\}  $ is a \textit{simplicial polyhedral
subdivision of }\ $\mathbb{R}_{+}^{n}$\ \textit{subordinate} to $\Gamma$, and
that $\left\{  \overline{\Delta}_{\tau,i}\right\}  $ is a \textit{simplicial
fan} \textit{subordinate} to $\Gamma$.

By adding new rays , each simplicial cone can be partitioned further into a
finite number of simple cones. In this way we obtain a \textit{simple
polyhedral subdivision} of $\mathbb{R}_{+}^{n}$\ \textit{subordinate} to
$\Gamma$, and a \textit{simple fan} \textit{subordinate} to $\Gamma$ (or a
\textit{complete regular fan}) (see e.g. \cite{K-M-S}).

Given a rational polyhedral subdivision $\Sigma^{\ast}$ of $\mathbb{R}_{+}%
^{n}$, we \ denote by Vert$\left(  \Sigma^{\ast}\right)  $ the set of all the
generators of the cones in $\Sigma^{\ast}$(this set is also called \textit{the
skeleton of} $\Sigma^{\ast}$). Each $n$-dimensional simple cone in
$\Sigma^{\ast}$ generated, say, by $a_{i}=\left(  a_{i,1},\ldots
,a_{i,n}\right)  \in$ Vert$\left(  \Sigma^{\ast}\right)  $, $i=1,\ldots,n$,
can be identified with a unimodular matrix $\left[  a_{1},\ldots,a_{n}\right]
$.

\subsection{The Newton polyhedron associated to an analytic mapping}

Let $\boldsymbol{f}=(f_{1},\ldots,f_{l})$, $\boldsymbol{f}\left(  0\right)  =0
$, be a non-constant polynomial mapping, or more generally, an analytic
mapping defined \ on a neighborhood $U\subseteq K^{n}$ of the origin. In this
paper we associate to $\boldsymbol{f}$ a Newton polyhedron $\Gamma\left(
\boldsymbol{f}\right)  :=\Gamma\left(
{\textstyle\prod\nolimits_{i=1}^{l}}
f_{i}\left(  x\right)  \right)  $. From a geometrical point of view,
$\Gamma\left(  \boldsymbol{f}\right)  $\ is the Minkowski sum of the
$\Gamma\left(  f_{i}\right)  $, for $i=1,\cdots,l$, (see e.g. \cite{O2},
\cite{St}). By using the results previously presented, we can associate to
$\Gamma\left(  \boldsymbol{f}\right)  $ a simple (or simplicial) polyhedral
subdivision $\Sigma^{\ast}\left(  \boldsymbol{f}\right)  $ of $\mathbb{R}%
_{+}^{n}$ subordinate to $\Gamma\left(  \boldsymbol{f}\right)  $. The
polyhedron $\Gamma\left(  \boldsymbol{f}\right)  $ is useful to construct
$\Sigma^{\ast}\left(  \boldsymbol{f}\right)  $, and then the corresponding
toric manifold. However, if $l>1$, the description of the real parts of the
poles of the local zeta functions defined in the introduction requires
$\Gamma\left(  \boldsymbol{f}\right)  $, and $\Gamma\left(  f_{j}\right)  $,
$j=1,\ldots,l$, as we will see later on.

\begin{remark}
\label{remark0}A basic fact\ about the Minkowski sum \ operation is the
additivity of the faces. From this fact follows:

\noindent(1) $F\left(  a,\Gamma\left(  \boldsymbol{f}\right)  \right)  =%
{\textstyle\sum\nolimits_{j=1}^{l}}
F\left(  a,\Gamma\left(  f_{j}\right)  \right)  $, for $a\in\mathbb{R}_{+}%
^{n}$\ ;

\noindent(2) $d\left(  a,\Gamma\left(  \boldsymbol{f}\right)  \right)  =%
{\textstyle\sum\nolimits_{j=1}^{l}}
d\left(  a,\Gamma\left(  f_{j}\right)  \right)  $, for $a\in\mathbb{R}_{+}%
^{n}$\ ;

\noindent(3) let $\tau\ $be a proper face of $\Gamma\left(  \boldsymbol{f}%
\right)  $, and let $\tau_{j}$ be proper face of $\Gamma\left(  f_{j}\right)
$, for $i=1,\cdots,l$. If $\tau=%
{\textstyle\sum\nolimits_{j=1}^{l}}
\tau_{j}$, then $\Delta_{\tau}\subseteq\overline{\Delta}_{\tau_{j}}$, for
$i=1,\cdots,l$.
\end{remark}

\begin{remark}
Note that the equivalence relation,
\[
a\sim a^{\prime}\Leftrightarrow F(a,\Gamma\left(  \boldsymbol{f}\right)
)=F(a^{\prime},\Gamma\left(  \boldsymbol{f}\right)  )\text{,}%
\]
used in the construction of a polyhedral subdivision of $\mathbb{R}_{+}^{n}%
$\ subordinate to\ $\Gamma\left(  \boldsymbol{f}\right)  $ can be equivalently
defined in the following form:%
\[
a\sim a^{\prime}\Leftrightarrow F(a,\Gamma\left(  f_{j}\right)  )=F(a^{\prime
},\Gamma\left(  f_{j}\right)  )\text{, for each }j=1,\ldots,l.
\]
This last definition is used in Oka's book \cite{O2}.
\end{remark}

\subsection{Non-degeneracy Conditions}

\begin{definition}
(1) Let $\boldsymbol{f}=(f_{1},\ldots,f_{l})$, $\boldsymbol{f}\left(
0\right)  =0$, be an analytic mapping defined \ on a neighborhood $U\subseteq
K^{n}$ of the origin. Let $\Gamma_{i}:=\Gamma\left(  f_{i}\right)  $ be the
Newton polyhedron of $f_{i}$ at the origin, for $i=1,\ldots,l$. The mapping
$\boldsymbol{f}$ is called \textit{non-degenerate with respect to} $\left(
\Gamma_{1},\ldots,\Gamma_{l}\right)  $ at the origin (or simply
\textit{non-degenerate}), if for every strictly positive vector $a\in
\mathbb{R}^{n}$ and any
\[
z\in\left\{  z\in\left(  K^{\times}\right)  ^{n}\cap U\mid f_{1,a}%
(z)=\ldots=f_{l,a}(z)=0\right\}  \text{,}%
\]
it satisfies that $rank_{K}\left[  \frac{\partial f_{i,_{a}}}{\partial x_{j}%
}\left(  z\right)  \right]  =\min\{l,n\}$.

\noindent(2) Let $V^{(j)}\left(  K\right)  =V^{(j)}=\left\{  z\in U\mid
f_{1}(z)=\ldots=f_{j}(z)=0\right\}  $, for $j=1,\ldots,l$, as before. If the
analytic mapping
\[%
\begin{array}
[c]{ccc}%
U & \rightarrow & K^{j}\\
x & \rightarrow & (f_{1}\left(  x\right)  ,\ldots,f_{j}\left(  x\right)  )
\end{array}
\]
is \textit{non-degenerate, we will say that }$V^{(j)}$ is \textit{a germ of
non-degenerate variety}.

\noindent(3) Let $\boldsymbol{f}=(f_{1},\ldots,f_{l})$, $\boldsymbol{f}\left(
0\right)  =0$, be a non-constant polynomial mapping. Let $\Gamma_{i}%
=\Gamma\left(  f_{i}\right)  $ be the Newton polyhedron of $f_{i}$ at the
origin, for $i=1,\ldots,l$. The mapping $\boldsymbol{f}$ is called
\textit{strongly non-degenerate with respect to} $\left(  \Gamma_{1}%
,\ldots,\Gamma_{l}\right)  $ (or simply \textit{strongly non-degenerate}), if
for every positive vector $a\in\mathbb{R}^{n}$, including the origin, and any
$z\in\left\{  z\in\left(  K^{\times}\right)  ^{n}\mid f_{1,a}(z)=\ldots
=f_{l,a}(z)=0\right\}  $, it satisfies that $rank_{K}\left[  \frac{\partial
f_{i,_{a}}}{\partial x_{j}}\left(  z\right)  \right]  =\min\{l,n\}$.

\noindent(4) Let $V^{(j)}\left(  K\right)  =V^{\left(  j\right)  }=\left\{
z\in K^{n}\mid f_{1}(z)=\ldots=f_{j}(z)=0\right\}  $, for $j=1,\ldots,l$. If
the mapping $x\rightarrow(f_{1}\left(  x\right)  ,\ldots,f_{j}\left(
x\right)  )$ is \textit{\ strongly non-degenerate, we will say that
}$V^{\left(  j\right)  }$ is \textit{a non-degenerate complete intersection
variety with respect to the coordinate system }$x=\left(  x_{1},\ldots
,x_{n}\right)  $.
\end{definition}

The above notion was introduced by Khovansky \cite{K}, see also \cite{O2}. For
a discussion about the relation between Khovansky's non-degeneracy notion and
other similar notions we refer the reader to \cite{V-Z}.

\begin{remark}
\label{Remark1}We want to keep the classical terminology used in singularity
theory. However, we will use this terminology in a more general setting. We
will use the following conventions. By a $K$-algebraic variety $W$, always
affine in this paper, we mean \ the set of $K$-rational points of $W$ with the
structure of $K$-analytic set. By a smooth $K$-algebraic variety $V$, we mean
the set of $K$-rational points of $V$ with the structure of $K$-analytic
closed submanifold. In particular the dimension of $V$ is the dimension of the
underlying $K$-analytic submanifold.
\end{remark}

\begin{definition}
An analytic function $f_{i}(x)$ is called convenient if for any $k=1,\ldots
,n$, there exists a monomial $x_{k}^{m_{k}}$ with non-zero coefficient in the
Taylor expansion of $f_{i}(x)$. An analytic mapping
\[%
\begin{array}
[c]{cccc}%
(f_{1},\ldots,f_{j}): & U & \rightarrow & K^{j\text{ }}\\
& x & \rightarrow & (f_{1}\left(  x\right)  ,\ldots,f_{j}\left(  x\right)  )
\end{array}
\]
is \ called \textit{\ convenient}, if each $f_{i}(x)$ is convenient for
$i=1,\ldots,j$. In this case we will say that \ $V^{\left(  j\right)  }$ is
convenient variety.
\end{definition}

We set $E_{i}:=\left(  0,\ldots,\underbrace{%
\begin{array}
[c]{c}%
1\\
i-\text{th place}%
\end{array}
},\ldots,0\right)  $, for $i=1,\ldots,n$. From now on, we will assume \ that
$\boldsymbol{f}=(f_{1},\ldots,f_{l})$ is convenient, thus there exists a
simple polyhedral subdivision of $\mathbb{R}_{+}^{n}$\ subordinate to
$\Gamma\left(  \boldsymbol{f}\right)  $ such that any \ not strictly positive
vector in Vert$\left(  \Sigma^{\ast}\right)  $ belongs to $\left\{
E_{1},\ldots,E_{n}\right\}  $.

\subsection{Resolution of Singularities of Non-degenerate Complete
Intersection Varieties}

The resolution of singularities for non-degenerate complete intersection
varieties is a well-known fact (see e.g. \cite{K}, \cite{Mo}, \cite{O2},
\cite{Var}, see also \cite{V-Z}, and references therein). The version of the
resolution of singularities needed here is a simple variation of the one given
in \cite[Theorem 3.4]{O2}, \cite[Proposition 2.5]{Mo}. For the material needed
to adapt the proof given in Oka%
\'{}%
s book to the $p$-adic setting we refer the reader to \cite{I2}, \cite{Ser}.
We now review the basic facts about resolution of singularities of
non-degenerate complete intersection varieties, without proofs, following
Oka's book. We warn the reader that our terminology and notation are slightly
different to the corresponding in \cite{O2}.

Let $V^{(j)}(K)$, $j=l-1,l$, be germs of convenient and non-degenerate
complete intersection varieties. Fix a rational simple polyhedral subdivision
$\Sigma^{\ast}=\Sigma^{\ast}\left(  \boldsymbol{f}\right)  $ \ of
$\mathbb{R}_{+}^{n}$ subordinate to $\Gamma\left(  \boldsymbol{f}\right)  $.
Let $X(K)$ be the toric manifold corresponding to $\Sigma^{\ast}$, and let
$h:X(K)\rightarrow U$ be the corresponding toric modification. We set $\left(
V^{\ast}\right)  ^{(j)}(K):=V^{(j)}(K)\cap\left(  K^{\times}\right)  ^{n}$,
$j=l-1,l$, and $V_{U}^{(j)}(K)$ for the closure of $\left(  V^{\ast}\right)
^{(j)}(K)$ in $U$, $j=l-1,l$. Let $\widetilde{V}^{(j)}(K)$\ be the strict
transform of $\left(  V^{\ast}\right)  ^{(j)}(K)$ in $X(K)$, for $j=l-1,l$. We
set $H^{(j)}(K):=\left\{  z\in U\mid f_{j}\left(  z\right)  =0\right\}  $,
$\left(  H^{\ast}\right)  ^{(j)}(K):=H^{(j)}(K)\cap\left(  K^{\times}\right)
^{n}$, and $H_{U}^{(j)}(K)$ for the closure of $\left(  H^{\ast}\right)
^{(j)}(K)$ in $U$, for $j=1,\ldots,l$. Let $\widetilde{H}^{(j)}(K)$ be the
strict transform of $\left(  H^{\ast}\right)  ^{(j)}(K)$ in $X(K)$, for
$j=1,\ldots,l$.

We recall that $h:X(K)\smallsetminus h^{-1}\left(  0\right)  \rightarrow
U\smallsetminus\left\{  0\right\}  $ is a $K$-analytic isomorphism, and that
\[
h^{-1}\left(  0\right)  =%
{\displaystyle\bigcup\limits_{\substack{a\in\text{Vert}(\Sigma^{\ast}%
)\\a\succ0}}}
E_{a}(K),
\]
where $E_{a}(K)$ is the compact exceptional divisor corresponding to $a$, see
\cite[Corollary 1.4.1]{O2}.

\begin{theorem}
[{\cite[Theorem 3.4]{O2}}]\label{proposition1}Assume that $V^{\left(
l\right)  }\left(  K\right)  $, $V^{\left(  l-1\right)  }\left(  K\right)
$\ are germs of convenient and non-degenerate complete intersection varieties
at the origin. Fix a rational simple polyhedral subdivision $\Sigma^{\ast
}=\Sigma^{\ast}\left(  \boldsymbol{f}\right)  $ of $\mathbb{R}_{+}^{n}$
subordinate to\ $\Gamma\left(  \boldsymbol{f}\right)  $. Let $X\left(
K\right)  $ be the toric manifold corresponding to $\Sigma^{\ast}$, and let
$h:X\left(  K\right)  \rightarrow U$ be the corresponding toric modification.
There exist a neighborhood \ $U_{0}\subseteq U$ of the origin such that the
following assertions are true over $U_{0}$. Put $X^{\prime}%
(K):=X(K)\smallsetminus h^{-1}\left(  0\right)  $ and $Y^{\prime
}(K):=U\smallsetminus\left\{  0\right\}  $.

\noindent(1) $h:\widetilde{V}^{(j)}(K)\rightarrow V_{U}^{(j)}(K)$ is a proper
mapping such that $h:\widetilde{V}^{(j)}(K)\cap X^{\prime}(K)\rightarrow
V_{U}^{(j)}(K)\cap Y^{\prime}(K)$ is a $K$-analytic isomorphism, for $j=l-1,l$.

\noindent(2) The divisor of the pullback function $h^{\ast}f_{j}$ is given by
\[
div\left(  h^{\ast}f_{j}\right)  (K)=\widetilde{H}^{(j)}(K)+\sum
_{\substack{a\in Vert\left(  \Sigma^{\ast}\right)  \\a\succ0}}d\left(
a,\Gamma_{j}\right)  E_{a}(K).
\]

\noindent(3) $\widetilde{V}^{(j)}(K)$ is a submanifold of $X(K)$ of dimension
$n-j$ so that $\widetilde{V}^{(j)}(K)=\cap_{i=1}^{j}\widetilde{H}^{(i)}(K)$,
$j=l-1,l$. In addition $\widetilde{V}^{(j)}(K)$, $j=l-1,l$, \ intersect the
exceptional divisor of $h$ transversely.

\noindent(4) $h:\widetilde{V}^{(l-1)}(K)\smallsetminus h^{-1}\left(  0\right)
\rightarrow V^{(l-1)}(K)\smallsetminus\left\{  0\right\}  $ is a $K$-analytic isomorphism.
\end{theorem}

\begin{remark}
\label{remark2A}

In the next section we use the previous theorem to give an explicit list for
\ the possible poles of $Z_{\Phi}(\omega,V^{(l-1)},f_{l})$. We fix some
notations needed for the next section. Let $\Delta\in\Sigma^{\ast}$ be an
$n$-dimensional simple cone generated by $a_{i}=\left(  a_{1,i},\ldots
,a_{n,i}\right)  $, $i=1,\ldots,n$. Then in the chart of $X\left(  K\right)  $
corresponding to $\Delta$, the map $h$ has the form%
\[%
\begin{array}
[c]{cccc}%
h: & K^{n} & \rightarrow & U\\
& y & \rightarrow & x,
\end{array}
\]
where $x_{i}=%
{\textstyle\prod\nolimits_{j=1}^{n}}
y_{j}^{a_{i,j}}$, with $\left[  a_{i,j}\right]  =\left[  a_{1},\ldots
,a_{n}\right]  $. Denote this chart by $W_{\Delta}$.

We note \ that since $h$ only maps to $U$ instead to the whole $K^{n}$, at
some charts it will not be defined \ everywhere on $K^{n}$. We set $I=\left\{
1,\ldots,r\right\}  $ with $r\leq n$, and
\[
T_{I}(K):=\left\{  y\in W_{\Delta}\mid y_{i}=0\Longleftrightarrow i\in
I\right\}  .
\]
We consider on $W_{\Delta}$ points on $h^{-1}\left(  0\right)  $ different
from the origin of $W_{\Delta}$. Let $b\in T_{I}(K)$ be one of these points.
We consider the following two cases: (I) $b\in T_{I}(K)\cap\widetilde{V}%
^{(l)}\left(  K\right)  $; (II) $b\in T_{I}(K)\cap\widetilde{V}^{(l-1)}\left(
K\right)  $ and $b\notin\widetilde{V}^{(l)}\left(  K\right)  $.
\end{remark}

\textbf{Case I}. In this case $b=\left(  0,\dots,0,b_{r+1},\dots,b_{n}\right)
$ with
\[
\widetilde{b}=(b_{r+1},\dots,b_{n})\in(K^{\times})^{n-r}.
\]
By Theorem \ref{proposition1}(2),
\begin{equation}
\left(  f_{i}\circ h\right)  \left(  y\right)  =\left(
{\displaystyle\prod\limits_{j=1}^{r}}
y_{j}^{d\left(  a_{j},\Gamma_{i}\right)  }\right)  \left(  \widetilde{f}%
_{i}\left(  y_{r+1},\ldots,y_{n}\right)  +O_{i}\left(  y_{1},\ldots
,y_{n}\right)  \right)  , \label{Eq14}%
\end{equation}
where the $O_{i}\left(  y_{1},\ldots,y_{n}\right)  $ are analytic functions
belonging to the ideal generated by $y_{1},\ldots,y_{r}$. By using that
\begin{equation}
\widetilde{b}\in\left(  K^{\times}\right)  ^{n-r}\cap\left\{  \widetilde
{f_{1}}(b_{r+1},\ldots,b_{n})=\dots=\widetilde{f_{l}}(b_{r+1},\ldots
,b_{n})=0\right\}  , \label{15}%
\end{equation}
and Theorem \ref{proposition1}(3), there exists a coordinate system
\[
y^{\prime}=(y_{1},\dots,y_{r},y_{r+1}^{\prime},\dots,y_{n}^{\prime})
\]
in a neighborhood $W_{b}$ of $b$ such that%
\begin{equation}
\left(  f_{i}\circ h\right)  \left(  y^{\prime}\right)  =\left(
{\displaystyle\prod\limits_{j=1}^{r}}
y_{j}^{d\left(  a_{j},\Gamma_{i}\right)  }\right)  y_{r+i}^{\prime},
\label{Eq20}%
\end{equation}
for $i=1,\ldots,l$, and
\[
\left\{  y_{r+1}^{\prime}=\ldots=y_{r+l}^{\prime}=0\right\}  \text{,
respectively }\left\{  y_{r+1}^{\prime}=\ldots=y_{r+l-1}^{\prime}=0\right\}
,
\]
is a local description in $W_{b}$ of $\widetilde{V}^{(l)}\left(  K\right)  $,
respectively of $\widetilde{V}^{(l-1)}\left(  K\right)  $.

\textbf{Case II}. This case is similar to the previous one except that
\[
\widetilde{b}\in\left(  K^{\times}\right)  ^{n-r}\cap\left\{  \widetilde
{f_{1}}(b_{r+1},\ldots,b_{n})=\dots=\widetilde{f_{l-1}}(b_{r+1},\ldots
,b_{n})=0\right\}  ,
\]
and (\ref{Eq20}) holds for $i=1,\ldots,l-1$, and
\begin{equation}
\left(  f_{l}\circ h\right)  \left(  y^{\prime}\right)  =\left(
{\displaystyle\prod\limits_{j=1}^{r}}
y_{j}^{d\left(  a_{j},\Gamma_{l}\right)  }\right)  u\left(  y^{\prime}\right)
, \label{20A}%
\end{equation}
with $\left\vert u\left(  y^{\prime}\right)  \right\vert _{K}=\left\vert
u\left(  b\right)  \right\vert _{K}$ for any $y^{\prime}\in W_{b}$, and
$\left\{  y_{r+1}^{\prime}=\ldots=y_{r+l-1}^{\prime}=0\right\}  $ is a local
description in $W_{b}$ of $\widetilde{V}^{(l-1)}\left(  K\right)  $.

\begin{remark}
\label{remark2B} (1) If we replace in Theorem \ref{proposition1} the condition
\textquotedblleft$V^{\left(  l\right)  }\left(  K\right)  $, $V^{\left(
l-1\right)  }\left(  K\right)  $ are germs of convenient non-degenerate
complete intersection varieties at the origin,\textquotedblright by
\textquotedblleft$V^{\left(  l\right)  }\left(  K\right)  $, $V^{\left(
l-1\right)  }\left(  K\right)  $ are convenient and non-degenerate complete
intersection varieties,\textquotedblright\textit{\ }and $U$ by $K^{n}$, with a
similar proof \ we obtain \ a \textit{global} version of Theorem
\ref{proposition1}, that is, the conclusions (1)-(4) \ are valid without the
condition \textquotedblleft there exists a neighborhood $U_{0}\subset U$ of
the origin.\textquotedblright\ In this case $\widetilde{V}^{(l)}\left(
K\right)  $ and $\widetilde{V}^{(l-1)}\left(  K\right)  $ may have components
that are disjoint with the exceptional divisor of $h$.

\noindent(2) The condition \textquotedblleft there exists a neighborhood
$U_{0}\subset U$ of the origin\textquotedblright\ may be replaced by
\textquotedblleft$U$\textit{ is a sufficiently small neighborhood of the
origin}.\textquotedblright
\end{remark}

\section{The Poles of the Meromorphic Continuation of $Z_{\Phi}(\omega
,V^{(l-1)},f_{l})$}

Given $a=\left(  a_{1},\ldots,a_{n}\right)  \in\mathbb{R}^{n}\setminus\left\{
0\right\}  $, we put $\sigma\left(  a\right)  :=a_{1}+\ldots+a_{n}$ and
$d\left(  a,\Gamma_{l}\right)  =\min_{m\in\Gamma_{l}}\left\langle
m,a\right\rangle $ as before. If $d\left(  a,\Gamma_{l}\right)  \neq0$, we
define
\[
\mathcal{P}\left(  a\right)  =\left\{  -\left(  \frac{\sigma\left(  a\right)
-%
{\displaystyle\sum\limits_{j=1}^{l-1}}
d\left(  a,\Gamma_{j}\right)  }{d\left(  a,\Gamma_{l}\right)  }\right)
+\frac{2\pi\sqrt{-1}k}{d\left(  a,\Gamma_{l}\right)  \log q},\text{\ }%
k\in\mathbb{Z}\right\}  .
\]

\begin{theorem}
\label{theorem1}Let $\boldsymbol{f}=(f_{1},\ldots,f_{l}):U\rightarrow K^{l}$,
$l\geq2$, $\boldsymbol{f}\left(  0\right)  =0$, be an analytic mapping defined
\ on a neighborhood $U\subseteq K^{n}$ of the origin. Assume that
$V^{(l)}\left(  K\right)  $ and $V^{(l-1)}\left(  K\right)  $ are germs of
convenient and non-degenerate complete intersection varieties, and that
$V^{(l-1)}\left(  K\right)  $ is a closed submanifold of $U$. Fix a rational
simple polyhedral subdivision $\Sigma^{\ast}=\Sigma^{\ast}\left(
\boldsymbol{f}\right)  $ of $\mathbb{R}_{+}^{n}$ subordinate to\ $\Gamma
\left(  \boldsymbol{f}\right)  $. If $U$ is sufficiently small and $\Phi$ is a
Bruhat-Schwartz function whose support is contained in $U$, then $Z_{\Phi
}(\omega,V^{(l-1)},f_{l})$ is a rational function of $t=\omega\left(
\pi\right)  $, and its poles belong to the set%
\[%
{\displaystyle\bigcup\limits_{\substack{a\in Vert\left(  \Sigma^{\ast}\right)
\\a\succ0}}}
\mathcal{P}\left(  a\right)  \cup\left\{  -1+\frac{2\pi\sqrt{-1}k}{\log
q}\text{, }k\in\mathbb{Z}\right\}  .
\]

\end{theorem}

\begin{proof}
We fix a rational simple polyhedral subdivision $\Sigma^{\ast}=\Sigma^{\ast
}\left(  \boldsymbol{f}\right)  $ of $\mathbb{R}_{+}^{n}$, and use the
notation introduced in Theorem \ref{proposition1} and Remark \ref{remark2A}.
By Theorem \ref{proposition1} (4),
\begin{equation}
h:\widetilde{V}^{(l-1)}(K)\smallsetminus h^{-1}\left(  0\right)  \rightarrow
V^{(l-1)}(K)\smallsetminus\left\{  0\right\}  \label{Eq3}%
\end{equation}
is a $K$-analytic isomorphism, where $h^{-1}\left(  0\right)  $ is the
exceptional divisor of $h$. We use (\ref{Eq3}) as a change of variables in
$Z_{\Phi}(\omega,V^{(l-1)},f_{l})$:%

\begin{align*}
Z_{\Phi}(\omega,V^{(l-1)},f_{l})  &  =\int\limits_{V^{(l-1)}\left(  K\right)
\setminus\left\{  0\right\}  }\Phi\left(  x\right)  \omega\left(
f_{l}(x)\right)  \mid\gamma_{GL}\left(  x\right)  \mid\\
&  =\int\limits_{\widetilde{V}^{(l-1)}(K)\smallsetminus h^{-1}\left(
0\right)  }\Phi^{\ast}\left(  y\right)  \omega\left(  f_{l}^{\ast}(y)\right)
\mid\gamma_{GL}^{\ast}\left(  y\right)  \mid
\end{align*}
for $\omega\in\Omega_{0}\left(  K^{\times}\right)  $, where $\mid\gamma
_{GL}\left(  x\right)  \mid$ is the measure induced on $V^{(l-1)}(K)$ by a
Gel'fand-Leray differential form, and $\gamma_{GL}^{\ast}\left(  y\right)  $
is the pullback of $\gamma_{GL}\left(  x\right)  $ by $h$.

Since $\Phi^{\ast}\left(  y\right)  $ has compact support, it is sufficient to
establish the theorem for integrals of type%
\[
I(\omega):=\int\limits_{\widetilde{V}^{(l-1)}(K)\smallsetminus h^{-1}\left(
0\right)  }\Theta\left(  y\right)  \omega\left(  f_{l}^{\ast}(y)\right)
\mid\gamma_{GL}^{\ast}\left(  y\right)  \mid,
\]
where $\Theta\left(  y\right)  $ is the characteristic function of a
neighborhood $W_{b}$ (an open compact set which may be schrinked when
necessary) of a point $b\in\widetilde{V}^{(l-1)}(K)$. Furthermore, we may
assume that $W_{b}=c+\pi^{e}R_{K}^{n}$, for some $c=$ $\left(  c_{1}%
,\ldots,c_{n}\right)  \in K^{n}$, and that $W_{b}\subset W_{\Delta}$, the
chart corresponding to a simple cone $\Delta\in$ $\Sigma^{\ast}$ generated by
$a_{i}=\left(  a_{1,i},\ldots,a_{n,i}\right)  $, $i=1,\ldots,n$,\ as in Remark
\ref{remark2A}. We take $I=\left\{  1,\ldots,r\right\}  $ with $r\leq n$, and
study the meromorphic continuation of $I(\omega)$ for the two cases considered
in Remark \ref{remark2A}.

\textbf{Case I} ($b\in T_{I}(K)\cap\widetilde{V_{l}}\left(  K\right)  $). In
this case $r+l\leq n$, and there exists a coordinate system $y^{\prime
}=\left(  y_{1},\ldots,y_{r},y_{r+1}^{\prime},\ldots,y_{n}^{\prime}\right)  $
in $W_{b}$ such that%
\[
\left(  f_{i}\circ h\right)  \left(  y^{\prime}\right)  =\left(
{\displaystyle\prod\limits_{j=1}^{r}}
y_{j}^{d\left(  a_{j},\Gamma_{i}\right)  }\right)  y_{r+i}^{\prime},
\]
for $i=1,\ldots,l$, $\widetilde{V}^{(l-1)}(K)=\left\{  y_{r+1}^{\prime}%
=\ldots=y_{r+l-1}^{\prime}=0\right\}  $, and
\[
h^{\ast}\left(
{\displaystyle\bigwedge\limits_{1\leq j\leq n}}
dx_{j}\right)  =\left(  \eta\left(  y\right)
{\displaystyle\prod\limits_{j=1}^{r}}
y_{j}^{\sigma\left(  a_{j}\right)  -1}\right)
{\displaystyle\bigwedge\limits_{1\leq j\leq r}}
dy_{j}\text{ }%
{\textstyle\bigwedge}
\text{ }%
{\textstyle\bigwedge\limits_{r+1\leq j\leq n}}
y_{j}^{\prime},
\]
where $\eta\left(  y\right)  $ is a unit of the local ring \ $\mathcal{O}%
_{X(K),b}$. By \ schrinking $W_{b}$ if necessary, we assume that $\left\vert
\eta\left(  y\right)  \right\vert _{K}=\left\vert \eta\left(  b\right)
\right\vert _{K}$, for any $y\in W_{b}$. The form $\gamma_{GL}^{\ast}$ on
$\widetilde{V}^{(l-1)}(K)$ is determined by the condition%
\begin{align*}
&  \gamma_{GL}^{\ast}\left(  y^{\prime}\right)  \text{ }%
{\textstyle\bigwedge}
\text{ }%
{\textstyle\bigwedge\limits_{i=1}^{l-1}}
d\left(  \left(
{\displaystyle\prod\limits_{j=1}^{r}}
y_{j}^{d\left(  a_{j},\Gamma_{i}\right)  }\right)  y_{r+i}^{\prime}\right) \\
&  =\left(
{\displaystyle\prod\limits_{j=1}^{r}}
y_{j}^{\sigma\left(  a_{j}\right)  -1}\right)
{\displaystyle\bigwedge\limits_{1\leq j\leq r}}
dy_{j}\text{ }%
{\textstyle\bigwedge}
\text{ }%
{\textstyle\bigwedge\limits_{r+1\leq j\leq n}}
dy_{j}^{\prime}.
\end{align*}
Since the left side of the previous formula equals
\[
\left(
{\displaystyle\prod\limits_{j=1}^{r}}
y_{j}^{\sum_{i=1}^{l-1}d\left(  a_{j},\Gamma_{i}\right)  }\right)  \gamma
_{GL}^{\ast}\left(  y^{\prime}\right)  \text{ }%
{\textstyle\bigwedge}
\text{ }%
{\textstyle\bigwedge\limits_{i=1}^{l-1}}
dy_{r+i}^{\prime}%
\]
on $\widetilde{V}^{(l-1)}(K)$, we can take $\gamma_{GL}^{\ast}\left(
y^{\prime}\right)  $ equal to
\begin{equation}
\left(
{\displaystyle\prod\limits_{j=1}^{r}}
y_{j}^{\sigma\left(  a_{j}\right)  -1-\sum_{i=1}^{l-1}d\left(  a_{j}%
,\Gamma_{i}\right)  }\right)
{\displaystyle\bigwedge\limits_{1\leq j\leq r}}
dy_{j}\text{ }%
{\textstyle\bigwedge}
\text{ }dy_{r+l}^{\prime}\text{ }%
{\textstyle\bigwedge}
\text{ }%
{\textstyle\bigwedge\limits_{r+l+1\leq j\leq n}}
dy_{j}^{\prime} \label{dif_form}%
\end{equation}
which is an analytic form, and then
\begin{equation}
\sigma\left(  a_{j}\right)  -1-\sum_{i=1}^{l-1}d\left(  a_{j},\Gamma
_{i}\right)  \geq0,\text{ for }j=1,\ldots,r. \label{positive}%
\end{equation}

We set $Y(K):=W_{b}\cap\left(  \widetilde{V}^{(l-1)}(K)\smallsetminus
h^{-1}\left(  0\right)  \right)  $,%

\[
\widetilde{\chi}(y^{\prime}):=\chi\left(  ac\text{ }\left(
{\displaystyle\prod\limits_{j=1}^{r}}
y_{j}^{d\left(  a_{j},\Gamma_{l}\right)  }\right)  \right)  \chi\left(
ac\text{ }y_{r+l}^{\prime}\right)  ,
\]
and
\[
\left\vert dy^{\prime\prime}\right\vert :=\left\vert
{\displaystyle\bigwedge\limits_{j=1}^{r}}
dy_{j}\text{ }%
{\textstyle\bigwedge}
\text{ }dy_{r+l}^{\prime}\text{ }%
{\textstyle\bigwedge}
\text{ }%
{\textstyle\bigwedge\limits_{j=r+l+1}^{n}}
dy_{j}^{\prime}\right\vert .
\]
We can express $I(\omega)$ as follows:
\[
\int\limits_{Y(K)}\widetilde{\chi}(y^{\prime})%
{\displaystyle\prod\limits_{j=1}^{r}}
\left(  \left\vert y_{j}\right\vert _{K}^{d\left(  a_{j},\Gamma_{l}\right)
s+\sigma\left(  a_{j}\right)  -1-\sum_{i=1}^{l-1}d\left(  a_{j},\Gamma
_{i}\right)  }\right)  \left\vert y_{r+l}^{\prime}\right\vert _{K}%
^{s}\left\vert dy^{\prime\prime}\right\vert ,
\]
i.e., $I(\omega)$ is the product (up to a constant) of the following two
integrals:%
\begin{equation}%
{\displaystyle\prod\limits_{j=1}^{r}}
\text{ }\int\limits_{c_{j}+\pi^{e}R_{K}}\chi\left(  ac\text{ }y_{j}\right)
^{d\left(  a_{j},\Gamma_{l}\right)  }\left\vert y_{j}\right\vert
_{K}^{d\left(  a_{j},\Gamma_{l}\right)  s+\sigma\left(  a_{j}\right)
-1-\sum_{i=1}^{l-1}d\left(  a_{j},\Gamma_{i}\right)  }\left\vert
dy_{j}\right\vert , \label{Eq4}%
\end{equation}
and
\begin{equation}
\int\limits_{c_{r+l}+\pi^{e}R_{K}}\chi\left(  ac\text{ }y_{r+l}^{\prime
}\right)  \left\vert y_{r+l}^{\prime}\right\vert _{K}^{s}\left\vert
dy_{r+l}^{\prime}\right\vert , \label{Eq5}%
\end{equation}
for $\operatorname{Re}(s)>0$. By applying Lemma 8.2.1 in \cite{I2}\ we obtain
the meromorphic continuation for integrals in (\ref{Eq4}) and (\ref{Eq5}).
Therefore the real parts of the poles of the meromorphic continuation of
$Z_{\Phi}(\omega,V^{(l-1)},f_{l})$\ have the form%
\[
-\left(  \frac{\sigma\left(  a\right)  -%
{\displaystyle\sum\limits_{j=1}^{l-1}}
d\left(  a,\Gamma_{j}\right)  }{d\left(  a,\Gamma_{l}\right)  }\right)
\text{, }a\in Vert\left(  \Sigma^{\ast}\right)  \text{ with }a\succ0\text{, or
}-1.
\]

For each strictly positive vertex $a$, we have an exceptional variety $E_{a}$.
The exceptional varieties $E_{a}$ \ which do not intersect $\widetilde
{V}^{(l-1)}(K)$ do certainly not induce a pole. For a vertex $a$ for which
$E_{a}$ intersects $\widetilde{V}^{(l-1)}(K)$, we already proved that
$\sigma\left(  a\right)  -\sum_{j=1}^{l-1}d\left(  a,\Gamma_{j}\right)  >0$
(cf. (\ref{positive})). This complete the description of the possible poles of
$Z_{\Phi}(\omega,V^{(l-1)},f_{l})$.

\textbf{Case II }($b\in T_{I}(K)\cap\widetilde{V}^{(l-1)}\left(  K\right)  $
and $b\notin\widetilde{V}^{(l)}\left(  K\right)  $). In this case $r+l-1\leq
n$. By using the same reasoning as in the previous case one gets that
$I(\omega)$ equals a constant multiplied by (\ref{Eq4}).
\end{proof}

\begin{remark}
\label{remark5} If in Theorem \ref{theorem1} we\ assume that $\boldsymbol{f}%
=(f_{1},\ldots,f_{l})$, $\boldsymbol{f}\left(  0\right)  =0$, is a polynomial
mapping, $U=K^{n}$, and $V^{(l)}\left(  K\right)  $ and $V^{(l-1)}\left(
K\right)  $ are convenient and non-degenerate intersection varieties, with
$V^{(l-1)}\left(  K\right)  $ a closed submanifold of $K^{n}$. Then the
conclusion of Theorem \ref{theorem1} holds without the condition
\textquotedblleft$U$ is sufficiently small.\textquotedblright
\end{remark}

\subsection{The Largest Real Part of the Poles of $Z_{\Phi}(s,V^{(l-1)}%
,f_{l})$}

Given a $\boldsymbol{f}:U\rightarrow K^{l}$, $\boldsymbol{f}\left(  0\right)
=0$, an analytic mapping defined \ on a neighborhood $U\subseteq K^{n}$ of the
origin, and a fixed rational simple polyhedral subdivision $\Sigma^{\ast
}=\Sigma^{\ast}\left(  \boldsymbol{f}\right)  $ of $\mathbb{R}_{+}^{n}$
subordinate to $\Gamma\left(  \boldsymbol{f}\right)  $, we set
\[
Z_{\Phi}(s,V^{(l-1)},f_{l}):=\int\limits_{V^{(l-1)}\left(  K\right)  }%
\Phi\left(  x\right)  \left\vert f_{l}(x)\right\vert _{K}^{s}\mid\gamma
_{GL}\left(  x\right)  \mid,
\]
for $\operatorname{Re}(s)>0$. As always we will identify $Z_{\Phi}%
(s,V^{(l-1)},f_{l})$ with its meromorphic continuation. The correspondence
\[%
\begin{array}
[c]{ccc}%
S(K^{n}) & \rightarrow & \mathbb{Q}\left(  q^{-s}\right) \\
&  & \\
\phi & \rightarrow & Z_{\Phi}(s,V^{(l-1)},f_{l}),
\end{array}
\]
defines a meromorphic distribution on $S(K^{n})$. By the poles of $Z_{\cdot
}(s,V^{(l-1)},f_{l})$ we mean the set $\cup_{\phi\in S(K^{n})}\left\{
\text{poles of }Z_{\phi}(s,V^{(l-1)},f_{l})\right\}  $. By using the fact that
$\boldsymbol{f}\left(  0\right)  =0$, and that $Z_{\phi}(s,V^{(l-1)},f_{l})$
can be expressed as a finite sum of Igusa's local zeta functions, it follows
from \cite[Lemma 2.6]{V-Z}\ that $Z_{\phi}(s,V^{(l-1)},f_{l})$\ has at least a pole.

We set $\beta_{\boldsymbol{f}}$ to be the largest real part of the poles of
$Z_{\cdot}(s,V^{(l-1)},f_{l})$, and $j_{\boldsymbol{f}}$ to be the maximal
order of the poles of $Z_{\cdot}(s,V^{(l-1)},f_{l})$\ having real part
$\beta_{\boldsymbol{f}}$. By abuse of language we will say that $\beta
_{\boldsymbol{f}}$ is the largest real part of the poles of $Z_{\phi
}(s,V^{(l-1)},f_{l})$. We also set $\gamma_{\boldsymbol{f}}$ \ to be the
maximum of the
\[
-\left(  \frac{\sigma\left(  a\right)  -%
{\displaystyle\sum\limits_{j=1}^{l-1}}
d\left(  a,\Gamma_{j}\right)  }{d\left(  a,\Gamma_{l}\right)  }\right)  ,
\]
where $a$ runs through all the strictly positive vectors in Vert$\left(
\Sigma^{\ast}\right)  $ satisfying $d\left(  a,\Gamma_{l}\right)  \neq0$
and$\ \sigma\left(  a\right)  -\sum_{j=1}^{l-1}d\left(  a,\Gamma_{j}\right)
>0$.

\begin{remark}
\label{remark5A} If $\gamma_{\boldsymbol{f}}$ $>-1$, then by Theorem
\ref{theorem1}, $\gamma_{\boldsymbol{f}}$ $\geq\beta_{\boldsymbol{f}}$ and
$j_{\boldsymbol{f}}$ $\leq n-l+1$.
\end{remark}

If $\boldsymbol{f}=(f_{1},\ldots,f_{l})$, $\boldsymbol{f}\left(  0\right)
=0$, is a polynomial mapping, $U=K^{n}$, and $V^{(l)}\left(  K\right)  $ and
$V^{(l-1)}\left(  K\right)  $ are convenient and non-degenerate complete
intersection varieties, with $V^{(l-1)}\left(  K\right)  $ a closed
submanifold of $K^{n}$, then, with the obvious analogous definitions for
$\beta_{\boldsymbol{f}}$ , $j_{\boldsymbol{f}}$ and $\gamma_{\boldsymbol{f}}$,
and deleting the condition \textquotedblleft strictly
positive\textquotedblright\ in the definition of $\gamma_{\boldsymbol{f}}$,
Remark \ref{remark5A}\ holds.

The largest real part of the poles of local zeta functions has been studied
intensively \cite{D1a}, \cite{D-H}, \cite{I1}, \cite{Var}, \cite{V-Z},
\cite{Z0}, \cite{Z1}. In the case $l>1$ the largest real part of the poles of
$Z_{\Phi}(s,V^{(l-1)},f_{l})$ is not completely determined by the $\Gamma_{i}$.

\subsection{Vanishing of $Z_{\Phi}(\omega,V^{(l-1)},f_{l})$}

Given $\lambda\in K^{\times}$, we set%
\[
V^{\left(  l,\lambda\right)  }\left(  K\right)  :=V^{\left(  l,\lambda\right)
}=\left\{  z\in U\mid f_{1}(z)=\ldots=f_{l-1}(z)=0\text{, }f_{l}%
(z)=\lambda\right\}  .
\]
Since \ any $\omega\in\Omega\left(  K^{\times}\right)  $ can be expressed as
$\omega\left(  z\right)  =\chi\left(  ac\text{ }z\right)  \left\vert
z\right\vert _{K}^{s}$, $s\in\mathbb{C}$, for $z\in K^{\times}$, we use the
notation $Z_{\Phi}(\omega,V^{(l-1)},f_{l}):=Z_{\Phi}(s,\chi,V^{(l-1)}%
,f_{l}):=Z_{\Phi}(s,\chi)$.

\begin{theorem}
\label{theorem2}Assume that the $l-$form $%
{\textstyle\bigwedge\nolimits_{i=1}^{l}}
df_{i}$, with $2\leq l\leq n$, does not vanish on $V^{\left(  l,\lambda
\right)  }$, for any $\lambda\in K^{\times}$. Then, with the hypotheses of
Theorem \ref{theorem1}, there exists $e\left(  \Phi\right)  >0$ in
$\mathbb{N}$ such that $Z_{\Phi}(s,\chi,V^{(l-1)},f_{l})=0$, for every
$s\in\mathbb{C}$, unless the conductor $c\left(  \chi\right)  $ of $\chi$
satisfies $c\left(  \chi\right)  \leq e\left(  \Phi\right)  $.
\end{theorem}

\begin{proof}
By using the proof of Proposition \ref{proposition0}, and all the notation
introduced there, \ we have that $Z_{\Phi}(\omega,V^{(l-1)},f_{l})$ can be
expressed as linear combination of classical Igusa's zeta functions, see
(\ref{Igusa}). The result follows from Theorem 8.4.1 in \cite{I2} by the
following assertion.

\textbf{Claim }$\nabla h\left(  y\right)  \neq0$, for any $y$ $\in$ support of
$\Theta$ $\cap$ $\left\{  h\left(  y\right)  \neq0\right\}  $.

The hypothesis $%
{\textstyle\bigwedge\nolimits_{i=1}^{l}}
df_{i}\neq0$ on $V^{\left(  l,\lambda\right)  }$, for any $\lambda\in
K^{\times}$, is equivalent to the matrix $\left[  \frac{\partial f_{i}%
}{\partial x_{j}}\left(  z\right)  \right]  $ has rank $l$, for any $z\in
V^{\left(  l,\lambda\right)  }\left(  K\right)  \setminus V^{\left(  l\right)
}\left(  K\right)  $. We take, as in the proof of Proposition
\ref{proposition0}, a \ point $b\in V^{(l-1)}\left(  K\right)  $ and a
coordinate system $y=\left(  y_{1},\ldots,y_{n}\right)  =\phi\left(  x\right)
$ around $b$ such that%
\[
V^{(l-1)}\left(  K\right)  =\left\{  y_{i}=0,\text{ }i=1,\ldots,l-1\right\}
,
\]
locally, and $h\left(  y\right)  =\left(  f_{l}\circ\phi^{-1}\right)  \left(
0,\ldots,0,y_{l},\ldots,y_{n}\right)  $, then%
\[
rank_{K}\left[  \frac{\partial f_{i}}{\partial x_{j}}\right]  =rank_{K}\left[
\begin{array}
[c]{ccc}%
I_{\left(  l-1\right)  \times\left(  l-1\right)  } &  & O_{\left(  l-1\right)
\times\left(  n-l+1\right)  }\\
&  & \\
O_{\left(  1\right)  \times\left(  l-1\right)  } & \ldots & \frac{\partial
h}{\partial y_{l}}\ldots\frac{\partial h}{\partial y_{n}}%
\end{array}
\right]  =l,
\]
where $I_{l-1\times l-1}$ is the identity matrix and $O_{\left(  l-1\right)
\times\left(  n-l+1\right)  }$, $O_{\left(  1\right)  \times\left(
l-1\right)  }$ are zero matrices, at any point of support of $\Theta$
$\cap\left\{  h\left(  y\right)  \neq0\right\}  $. Hence, at any point
$y_{0}\in$support of $\Theta$ $\cap$ $\left\{  h\left(  y\right)
\neq0\right\}  $, there exists $i_{0}\in\left\{  l,\ldots,n\right\}  $ such
that $\frac{\partial h}{\partial y_{i_{0}}}\left(  y_{0}\right)  \neq0$.
\end{proof}

\begin{remark}
\label{remark6} The conclusion in Theorem \ref{theorem2} holds, if the
condition \textquotedblleft the hypotheses of Theorem \ref{theorem1}%
\textquotedblright\ is replaced by \textquotedblleft the hypotheses of Remark
\ref{remark5}.\textquotedblright
\end{remark}

\subsection{The Oscillatory Integrals $E_{\Phi}(z)$}

In this section we study the asymptotic behavior of the oscillatory integral
defined in the introduction:%
\[
E_{\Phi}(z,V^{(l-1)},f_{l})=E_{\Phi}(z)=\int\limits_{V^{(l-1)}\left(
K\right)  }\Phi\left(  x\right)  \Psi\left(  zf_{l}(x)\right)  \mid\gamma
_{GL}\left(  x\right)  \mid,
\]
where $z=u\pi^{-m}$, $u\in R_{K}^{\times}$, $m\in\mathbb{Z}$, and $\Psi\left(
\cdot\right)  $ is the standard additive character on $K$.

Let Coeff$_{t^{k}}Z_{\Phi}(s,\chi)$ denote the coefficient $c_{k}$ in the
power expansion of $Z_{\Phi}(s,\chi)$ in the variable $t=q^{-s}.$

\begin{proposition}
\label{proposition3}Assume that the $l-$form $%
{\textstyle\bigwedge\nolimits_{i=1}^{l}}
df_{i}$, with $2\leq l\leq n$, does not vanish on $V^{(l,\lambda)}$, for any
$\lambda\in K^{\times}$. Then, with the hypotheses of Theorem \ref{theorem1},%
\begin{align*}
E_{\Phi}\left(  u\pi^{-m}\right)   &  =Z_{\Phi}^{\left(  l\right)  }%
(0,\chi_{\text{triv}})+\text{Coeff}_{t^{m-1}}\frac{\left(  t-q\right)
Z_{\Phi}(s,\chi_{\text{triv}})}{\left(  q-1\right)  \left(  1-t\right)  }+\\
&
{\displaystyle\sum\limits_{\chi\neq\chi_{\text{triv}}}}
g_{\chi^{-1}}\chi\left(  u\right)  \text{Coeff}_{t^{m-c\left(  \chi\right)  }%
}Z_{\Phi}^{\left(  l\right)  }(s,\chi),
\end{align*}
where $c\left(  \chi\right)  $\ denotes the conductor of $\chi$, and $g_{\chi
}$ denotes the Gaussian sum%
\[
g_{\chi}=\left(  q-1\right)  ^{-1}q^{1-c\left(  \chi\right)  }%
{\displaystyle\sum\limits_{v\in\left(  R_{K}/P_{K}^{c\left(  \chi\right)
}\right)  ^{\times}}}
\chi\left(  v\right)  \Psi\left(  v/\pi^{c\left(  \chi\right)  }\right)  .
\]

\end{proposition}

\begin{proof}
The proof uses the same reasoning as the one given by Denef for Proposition
1.4.4 in \cite{D0}.
\end{proof}

\begin{remark}
\label{remark9}The conclusion in Proposition \ref{proposition3} holds, if the
condition \textquotedblleft the hypotheses of Theorem \ref{theorem1}%
\textquotedblright\ is replaced by \textquotedblleft the hypotheses of Remark
\ref{remark5}.\textquotedblright
\end{remark}

\begin{theorem}
\label{theorem3}Let $\boldsymbol{f}=(f_{1},\ldots,f_{l}):U\rightarrow K^{l}$,
$\boldsymbol{f}\left(  0\right)  =0$, $2\leq l\leq n$, be an analytic mapping
defined \ on a neighborhood $U\subseteq K^{n}$ of the origin. Assume that
$V^{(l)}\left(  K\right)  $ and $V^{(l-1)}\left(  K\right)  $ are germs of
convenient and non-degenerate complete intersection varieties, and that
$V^{(l-1)}\left(  K\right)  $ is a closed submanifold of $U$. Fix a rational
simple polyhedral subdivision $\Sigma^{\ast}=\Sigma^{\ast}\left(
\boldsymbol{f}\right)  $ of $\mathbb{R}_{+}^{n}$ subordinate to $\Gamma\left(
\boldsymbol{f}\right)  $. Assume that $U$ is sufficiently small and $\Phi$ is
a Bruhat-Schwartz function whose support is contained in $U$. Assume that the
$l-$form $%
{\textstyle\bigwedge\nolimits_{i=1}^{l}}
df_{i}$ does not vanish on $V^{(l,\lambda)}\left(  K\right)  $, for any
$\lambda\in K^{\times}$. Then

\noindent(1) for $\left\vert z\right\vert _{K}$ big enough $E_{\Phi}(z)$ is a
finite $\mathbb{C}-$linear combination of the functions of the form
$\chi\left(  ac\text{ }z\right)  \left\vert z\right\vert _{K}^{\lambda}\left(
\log_{q}\left\vert z\right\vert _{K}\right)  ^{j_{\lambda}}$ with coefficients
independent of $z$, and $\lambda\in\mathbb{C}$ a pole of $\left(
1-q^{-s-1}\right)  Z_{\Phi}(s,\chi_{\text{triv}})$ or of $Z_{\Phi}(s,\chi)$,
$\chi\neq\chi_{\text{triv}}$, and with $j_{\lambda}\leq$(multiplicity of pole
$\lambda$ $-1$). Moreover all the poles $\lambda$\ appear effectively in this
linear combination;

\noindent(2) if $\gamma_{\boldsymbol{f}}>-1$, then $\left\vert E_{\Phi
}(z)\right\vert \leq C\left(  K\right)  \left\vert z\right\vert _{K}%
^{\gamma_{\boldsymbol{f}}}\left(  \log_{q}\left\vert z\right\vert _{K}\right)
^{n-l}$, for $\left\vert z\right\vert _{K}$ big enough, where \ $C\left(
K\right)  $\ is a positive constant.
\end{theorem}

\begin{proof}
(1) The result follows from Theorems \ref{theorem1}, \ref{theorem2}\ and
Proposition \ref{proposition3} by writing $Z_{\Phi}(s,\chi)$ in partial
fractions. (2) The estimation follows from the first part and Remark
\ref{remark5A}.
\end{proof}

\begin{remark}
\label{remark10}The conclusion in Theorem \ref{theorem3} holds, if the
condition \textquotedblleft the hypotheses of Theorem \ref{theorem1}%
\textquotedblright\ is replaced by \textquotedblleft the hypotheses of Remark
\ref{remark5}.\textquotedblright
\end{remark}

\section{\label{Congruences}Congruences and Exponential Sums Along Smooth
Algebraic Varieties}

For any polynomial $g$ over $R_{K}$ we denote by $\overline{g}$ \ the
polynomial over $\overline{K}$ obtained by reducing each coefficient of $g$
modulo $P_{K}$.

Assume that $f_{i}\left(  x\right)  \in R_{K}\left[  x_{1},\ldots
,x_{n}\right]  $, $f_{i}\left(  0\right)  =0$, for $i=1,\ldots,l$, with $2\leq
l\leq n$, and that $V^{(l-1)}\left(  K\right)  =\left\{  x\in K^{n}\mid
f_{i}\left(  x\right)  =0,\text{ }i=1,\ldots,l-1\right\}  $ is a closed
submanifold of dimension $n-l+1$. We will say that $V^{(l-1)}\left(  K\right)
$ is \textit{a} \textit{smooth }$K$\textit{-algebraic variety of dimension
}$n-l+1$, following the convention \ introduced in Remark \ref{Remark1}.

Since $R_{K}^{n}$ is compact,
\[
V^{(l-1)}\left(  R_{K}\right)  :=V^{(l-1)}\left(  K\right)  \cap R_{K}%
^{n}=\left\{  x\in R_{K}^{n}\mid f_{i}\left(  x\right)  =0,\text{ }%
i=1,\ldots,l-1\right\}
\]
is a compact submanifold of dimension $n-l+1$.

Let $\operatorname{mod}$ $P_{K}^{m}$ denote the canonical homomorphism
$R_{K}^{n}\rightarrow\left(  R_{K}/P_{K}^{m}\right)  ^{n}$, for $m,n\in
\mathbb{N\smallsetminus}\left\{  0\right\}  $. We will call the image of
$A\subseteq R_{K}^{n}$ by $\operatorname{mod}$ $P_{K}^{m}$, \textit{the
reduction }$\operatorname{mod}$\textit{\ }$P_{K}^{m}$\textit{\ of }$A$, and it
will be denoted as $A$ $\operatorname{mod}$ $P_{K}^{m}$.

We set
\[
V^{(l-1)}\left(  \overline{K}\right)  :=\left\{  \overline{z}\in\overline
{K}^{n}\mid\overline{f_{i}}\left(  \overline{z}\right)  =0,\text{ }%
i=1,\ldots,l-1\right\}  .
\]

We will say that $V^{(l-1)}\left(  K\right)  $ \textit{has good reduction}
$\operatorname{mod}$ $P_{K}$ if $rank_{\overline{K}}\left[  \overline
{\frac{\partial f_{i}}{\partial x_{j}}}\left(  \overline{z}\right)  \right]
=l-1$, for every $\overline{z}\in$ $V^{(l-1)}\left(  \overline{K}\right)  $.
In this case, we will say that $V^{(l-1)}\left(  K\right)  $ is \textit{a}
\textit{smooth }$K$\textit{-algebraic variety of dimension }$n-l+1$\textit{
with good reduction} $\operatorname{mod}$ $P_{K}$.

We also define for $m\in\mathbb{N\smallsetminus}\left\{  0\right\}  $,%
\[
V^{(l-1)}\left(  R_{K}/P_{K}^{m}\right)  =\left\{  \widetilde{x}\in\left(
R_{K}/P_{K}^{m}\right)  ^{n}\mid ord\left(  f_{i}\left(  \widetilde{x}\right)
\right)  \geq m,\text{ }i=1,\ldots,l-1\right\}  .
\]
We note that \textquotedblleft$ord\left(  f_{i}\left(  x\right)  \right)  \geq
m $\textquotedblright\ is independent of the representative chosen to compute
$ord\left(  f_{i}\left(  \widetilde{x}\right)  \right)  $.

\begin{definition}
(1) Let $f_{i}\left(  x\right)  \in R_{K}\left[  x_{1},\ldots,x_{n}\right]  $,
$f_{i}\left(  0\right)  =0$, for $i=1,\ldots,l$, with $2\leq l\leq n$. The
mapping $\boldsymbol{f}=(f_{1},\ldots,f_{l}):K^{n}\rightarrow K^{l}$, \ is
called strongly non-degenerate with respect to $\left(  \Gamma_{1}%
,\ldots,\Gamma_{l}\right)  $ (or simply strongly \textit{non-degenerate}) over
$\overline{K}$, if for every positive vector $a\in\mathbb{R}^{n}$, including
the origin, and any
\[
\overline{z}\in\left\{  \overline{z}\in\left(  \overline{K}^{\times}\right)
^{n}\mid\overline{f}_{1,a}(\overline{z})=\ldots=\overline{f}_{l,a}%
(\overline{z})=0\right\}  ,
\]
it satisfies that $rank_{\overline{K}}\left[  \overline{\frac{\partial
f_{i,_{a}}}{\partial x_{j}}}\left(  \overline{z}\right)  \right]  =l$.
Analogously we call $\boldsymbol{f}$\ strongly non-degenerate with respect to
$\left(  \Gamma_{1},\ldots,\Gamma_{l}\right)  $ at the origin (or simply
strongly \textit{non-degenerate at the origin}) over $\overline{K}$, if the
same condition is satisfied but only for $a$ strictly positive.

\noindent(2) Let $V^{\left(  j\right)  }\left(  K\right)  =V^{(j)}=\left\{
z\in K^{n}\mid f_{1}(z)=\ldots=f_{j}(z)=0\right\}  $, for $j=1,\ldots,l$, as
before. If the mapping $\boldsymbol{f}=\left(  f_{1},\ldots,f_{j}\right)  $
is\textit{\ strongly non-degenerate over }$\overline{K}$\textit{, we will say
that }$V^{(j)}$ is \textit{a } \textit{non-degenerate complete intersection
variety over }$\overline{K}$. If $\boldsymbol{f}=\left(  f_{1},\ldots
,f_{j}\right)  $\ is strongly non-degenerate at the origin over $\overline{K}%
$, we will say that $V^{(j)}$ is \textit{a } \textit{non-degenerate complete
intersection variety at the origin over }$\overline{K}$.
\end{definition}

We warn the reader that the main role of the word `strongly' in the previous
definition is to emphasize that we are working with a polynomial mapping and
that $U=K^{n}$.

For $a\in\mathbb{R}_{+}^{n}$ we set%
\[
V_{a}^{(j)}(R_{K}/P_{K}^{m}):=\left\{  x\in(R_{K}/P_{K}^{m})^{n}\mid
ord\left(  f_{i,a}(x)\right)  \geq m\text{, }i=1,\ldots,j\right\}  ,
\]
and%
\[
V_{a}^{(j)}(K):=\left\{  x\in K^{n}\mid f_{i,a}(x)=0\text{, }i=1,\ldots
,j\right\}  ,
\]
for $j=1,\ldots,l$. Analogously we define $V_{a}^{(j)}(\overline{K})$.

\subsection{Some Integrals Involving the Dirac Delta Function}

In this section, all the integrals involving the Dirac Delta function are
understood as defined in Gel'fand \ and Shilov's book \cite{G-S}, see also
Section \ref{DeltaSect}.

\begin{lemma}
\label{lemma1A}Let $f_{i}\left(  x\right)  \in R_{K}\left[  x_{1},\ldots
,x_{n}\right]  $, $f_{i}\left(  0\right)  =0$, for $i=1,\ldots,l$, with $2\leq
l\leq n$. Assume that $V^{(l)}\left(  K\right)  $ and $V^{(l-1)}\left(
K\right)  $ are non-degenerate complete intersection varieties over
$\overline{K}$. Let $x_{0}\in$ $\left(  R_{K}^{\times}\right)  ^{n}$ be a
given point, let $a\in\mathbb{R}_{+}^{n}$, and let $m\in\mathbb{N\setminus
}\left\{  0\right\}  $. We set
\[
I\left(  s,x_{0},m,a\right)  :=%
{\displaystyle\int\limits_{x_{0}+\left(  P_{K}^{m}\right)  ^{n}}}
\delta\left(  f_{1,a}\left(  x\right)  ,\ldots,f_{l-1,a}\left(  x\right)
\right)  \left\vert f_{l,a}\left(  x\right)  \right\vert _{K}^{s}\left\vert
dx\right\vert \text{, \ for }\operatorname{Re}(s)>0.
\]
Then $I\left(  s,x_{0},m,a\right)  $ equals
\[
\left\{
\begin{array}
[c]{lll}%
0 & \text{if\qquad} & \widetilde{x_{0}}\notin V_{a}^{(l-1)}\left(  R_{K}%
/P_{K}^{m}\right) \\
&  & \\
q^{-m\left(  n-l+1\right)  -sk} & \text{if\qquad} &
\begin{array}
[c]{l}%
\widetilde{x_{0}}\in V_{a}^{(l-1)}\left(  R_{K}/P_{K}^{m}\right)  \text{
and}\\
k:=ord\left(  f_{l,a}\left(  x_{0}\right)  \right)  <m
\end{array}
\\
&  & \\
q^{-m\left(  s+n-l+1\right)  }\left(  \frac{1-q^{-1}}{1-q^{-s-1}}\right)  &
\text{if\qquad} & \widetilde{x_{0}}\in V_{a}^{(l)}\left(  R_{K}/P_{K}%
^{m}\right)  ,
\end{array}
\right.
\]
where $\widetilde{x_{0}}$\ denotes the image of $x_{0}$\ in $R_{K}/P_{K}^{m}$.
\end{lemma}

\begin{proof}
Note that $\widetilde{x_{0}}\notin V_{a}^{(l-1)}\left(  R_{K}/P_{K}%
^{m}\right)  $ implies that $I\left(  s,x_{0},m,a\right)  =0$. We consider the
case $\widetilde{x_{0}}\in V_{a}^{(l-1)}\left(  R_{K}/P_{K}^{m}\right)  $ and
$k=ord\left(  f_{l,a}\left(  x_{0}\right)  \right)  <m$. By the Hensel lemma
we may assume $x_{0}\in V_{a}^{(l-1)}\left(  R_{K}\right)  $.Then $I\left(
s,x_{0},m,a\right)  $ can be expressed \ as%
\[
I\left(  s,x_{0},m,a\right)  =q^{-mn-sk}%
{\displaystyle\int\limits_{R_{K}^{n}}}
\delta\left(  f_{1,a}\left(  x_{0}+\pi^{m}x\right)  ,\ldots,f_{l-1,a}\left(
x_{0}+\pi^{m}x\right)  \right)  \left\vert dx\right\vert .
\]
By reordering the $x_{i}$%
\'{}%
s, \ and using the fact that $(f_{1},\ldots,f_{l-1})$ is strongly
non-degenerate over $\overline{K}$, we assume that $rank_{\overline{K}}\left[
\frac{\partial\overline{f_{i,a}}}{\partial x_{j}}\left(  \overline{z}\right)
\right]  =l-1$, for any $\overline{z}\in\left(  \overline{K}^{\times}\right)
^{n}$. We set $y=\left(  y_{1},\ldots,y_{n}\right)  =\phi\left(  x\right)  $
with
\[
y_{i}=\left\{
\begin{array}
[c]{ll}%
\frac{f_{i,a}\left(  x_{0}+\pi^{m}x\right)  }{\pi^{m}}, & i=1,\ldots,l-1\\
& \\
x_{i}, & i=l,\ldots,n.
\end{array}
\right.
\]
By using the implicit function theorem (see e.g. \cite[Lemma 7.4.3]{I2}), one
gets that $y=\phi\left(  x\right)  $ is measure-preserving bianalytic mapping
of $R_{K}^{n}$\ onto itself. By using $y=\phi\left(  x\right)  $ as change of
variables in $I\left(  s,x_{0},m,a\right)  $ one gets%
\[
I\left(  s,x_{0},m,a\right)  =q^{-mn-sk}%
{\displaystyle\int\limits_{R_{K}^{n}}}
\delta\left(  \pi^{m}y_{1},\ldots,\pi^{m}y_{l-1}\right)  \left\vert
dy\right\vert =q^{-m\left(  n-l+1\right)  -sk}.
\]

Finally we consider $\widetilde{x_{0}}\in V_{a}^{(l-1)}\left(  R_{K}/P_{K}%
^{m}\right)  $ and $k=ord\left(  f_{l,a}\left(  x_{0}\right)  \right)  \geq
m$. This condition is equivalent to $\widetilde{x_{0}}\in V_{a}^{(l)}\left(
R_{K}/P_{K}^{m}\right)  $, and by the Hensel lemma we may assume $x_{0}\in
V_{a}^{(l)}\left(  R_{K}\right)  $. By using a reasoning similar to the
previously done, one gets that
\begin{align*}
I\left(  s,x_{0},m,a\right)   &  =q^{-mn}%
{\displaystyle\int\limits_{R_{K}^{n}}}
\delta\left(  \pi^{m}y_{1},\ldots,\pi^{m}y_{l-1}\right)  \left\vert \pi
^{m}y_{l}\right\vert ^{s}\left\vert dy\right\vert \\
&  =q^{-mn-ms}\left(
{\displaystyle\int\limits_{R_{K}^{l-1}}}
\delta\left(  \pi^{m}y_{1},\ldots,\pi^{m}y_{l-1}\right)  \left\vert
dy\right\vert \right)  \left(
{\displaystyle\int\limits_{R_{K}}}
\left\vert y_{l}\right\vert ^{s}\left\vert dy_{l}\right\vert \right) \\
&  =q^{-m\left(  s+n-l+1\right)  }\left(  \frac{1-q^{-1}}{1-q^{-s-1}}\right)
.
\end{align*}

\end{proof}

\subsection{Polynomial Congruences over Submanifolds}

Along this section we will assume that $V^{(l-1)}\left(  K\right)  $ is a
smooth $K$-algebraic variety of dimension $n-l+1$ with good reduction
$\operatorname{mod}$ $P_{K}$. We set $N_{m}\left(  f_{l},V^{(l-1)}\right)
:=N_{m}$ as
\[
\left\{
\begin{array}
[c]{ll}%
card\left(  \left\{  \widetilde{x}\in V^{(l-1)}\left(  R_{K}\right)  \text{
}\operatorname{mod}P_{K}^{m}\mid ord\left(  f_{l}\left(  \widetilde{x}\right)
\right)  \geq m\right\}  \right)  , & \text{\ }m\geq1\\
& \\
1, & \text{\ }m=0.
\end{array}
\right.
\]

Since $V^{(l-1)}\left(  K\right)  $ has good reduction $\operatorname{mod}$
$P_{K}$, the Hensel lemma implies that
\[
V^{(l-1)}\left(  R_{K}\right)  \text{ }\operatorname{mod}P_{K}^{m}%
=V^{(l-1)}\left(  R_{K}/P_{K}^{m}\right)  ,
\]
and then
\begin{align*}
&  \left\{  \widetilde{x}\in V^{(l-1)}\left(  R_{K}\right)  \text{
}\operatorname{mod}P_{K}^{m}\mid ord\left(  f_{l}\left(  \widetilde{x}\right)
\right)  \geq m\right\} \\
& \\
&  =\left\{  \widetilde{x}\in\left(  R_{K}/P_{K}^{m}\right)  ^{n}\mid
f_{1}\left(  \widetilde{x}\right)  \equiv f_{2}\left(  \widetilde{x}\right)
\equiv\ldots\equiv f_{l}\left(  \widetilde{x}\right)  \equiv0\text{
}\operatorname{mod}P_{K}^{m}\right\}  .
\end{align*}
We associate to the sequence $\left(  N_{m}\right)  _{m\in\mathbb{N}}$ the
Poincar\'{e} series $P\left(  t,f_{l},V_{l-1}\right)  \allowbreak:=P\left(
t\right)  $ defined as%
\[
P\left(  t\right)  =%
{\displaystyle\sum\limits_{m=0}^{\infty}}
q^{-m\left(  n-l+1\right)  }N_{m}t^{m}.
\]
We also set
\[
Z\left(  s,V^{(l-1)},f_{l}\right)  :=\int\limits_{R_{K}^{n}}\delta\left(
f_{1}\left(  x\right)  ,\ldots,f_{l-1}\left(  x\right)  \right)  \left\vert
f_{l}(x)\right\vert _{K}^{s}\mid dx\mid.
\]

\begin{lemma}
\label{lemma2} If $t=q^{-s}$, then
\begin{equation}
P\left(  t\right)  =\frac{1-tZ\left(  s,V^{(l-1)},f_{l}\right)  }{1-t}.
\label{identity}%
\end{equation}

\end{lemma}

\begin{proof}
We first note that%
\[
Z\left(  s,V^{(l-1)},f_{l}\right)  =\int\limits_{R_{K}^{n}}\delta\left(
f_{1}\left(  x\right)  ,\ldots,f_{l-1}\left(  x\right)  \right)  \left\vert
f_{l}(x)\right\vert _{K}^{s}\mid dx\mid=
\]%
\begin{align*}
&
{\displaystyle\sum\limits_{m=0}^{\infty}}
q^{-ms}\left\{
{\displaystyle\int\limits_{\left\{  x\in R_{K}^{n}\mid ord\left(  f_{l}\left(
x\right)  \right)  \geq m\right\}  }}
\delta\left(  f_{1}\left(  x\right)  ,\ldots,f_{l-1}\left(  x\right)  \right)
\left\vert dx\right\vert \right\} \\
&  -%
{\displaystyle\sum\limits_{m=0}^{\infty}}
q^{-ms}\left\{
{\displaystyle\int\limits_{\left\{  x\in R_{K}^{n}\mid\left(  ordf_{l}\left(
x\right)  \right)  \geq m+1\right\}  }}
\delta\left(  f_{1}\left(  x\right)  ,\ldots,f_{l-1}\left(  x\right)  \right)
\left\vert dx\right\vert \right\}  .
\end{align*}
The result follows from the previous identity by using the following claim.

\textbf{Claim 1}%

\[%
{\displaystyle\int\limits_{\left\{  x\in R_{K}^{n}\mid ord\left(  f_{l}\left(
x\right)  \right)  \geq m\right\}  }}
\delta\left(  f_{1}\left(  x\right)  ,\ldots,f_{l-1}\left(  x\right)  \right)
\left\vert dx\right\vert =q^{-m\left(  n-l+1\right)  }N_{m}.
\]
The previous integral is equal to a finite sum of integrals of the form%

\[
I(x_{0},m):=%
{\displaystyle\int\limits_{\left\{  x\in x_{0}+\left(  P_{K}^{m}\right)
^{n}\mid ord\left(  f_{l}\left(  x\right)  \right)  \geq m\right\}  }}
\delta\left(  f_{1}\left(  x\right)  ,\ldots,f_{l-1}\left(  x\right)  \right)
\left\vert dx\right\vert
\]%
\[
=q^{-mn}%
{\displaystyle\int\limits_{\left\{  x\in R_{K}^{n}\mid ord\left(  f_{l}\left(
x_{0}+\pi^{m}x\right)  \right)  \geq m\right\}  }}
\delta\left(  f_{1}\left(  x_{0}+\pi^{m}x\right)  ,\ldots,f_{l-1}\left(
x_{0}+\pi^{m}x\right)  \right)  \left\vert dx\right\vert ,
\]
where $x_{0}\in R_{K}^{n}$ runs through a fixed set of representatives \ of
$V^{(l-1)}\left(  R_{K}\right)  $ $\operatorname{mod}P_{K}^{m}$. We may assume
that $x_{0}\in V^{(l-1)}\left(  R_{K}\right)  $. Indeed, we can choose another
set of representatives of $V^{(l-1)}(R_{K})$ $\operatorname{mod}$ $P_{K}^{m}$
which are in $V^{(l-1)}(R_{K})$ because $V^{(l-1)}(R_{K})$ has good reduction
$\operatorname{mod}$ $P_{K}$.

Note that $I(x_{0},m)=0$ if $ord\left(  f_{l}\left(  x_{0}\right)  \right)
\allowbreak<m$, and if $ord\left(  f_{l}\left(  x_{0}\right)  \right)  \geq
m$,
\[
I(x_{0},m)=q^{-mn}%
{\displaystyle\int\limits_{R_{K}^{n}}}
\delta\left(  f_{1}\left(  x_{0}+\pi^{m}x\right)  ,\ldots,f_{l-1}\left(
x_{0}+\pi^{m}x\right)  \right)  \left\vert dx\right\vert .
\]

By using the fact that $V^{(l-1)}\left(  K\right)  $ is a smooth $K$-algebraic
variety of dimension $n-l+1$ with good reduction $\operatorname{mod}$ $P_{K}$,
it follows from the implicit function theorem, possibly after reordering the
$x_{i}$'s, that $y=\phi\left(  x\right)  $, with
\[
y_{i}:=\left\{
\begin{array}
[c]{ll}%
\frac{f_{i}\left(  x_{0}+\pi^{m}x\right)  -f_{i}\left(  x_{0}\right)  }%
{\pi^{m}}, & i=1,\ldots,l-1\\
& \\
x_{i}, & i=l,\ldots,n,
\end{array}
\right.
\]
is measure-preserving bianalytic mapping of $R_{K}^{n}$\ onto itself.
Therefore
\[
I(x_{0},m)=q^{-mn}%
{\displaystyle\int\limits_{R_{K}^{l-1}}}
\delta\left(  \pi^{m}y_{1}+f_{1}\left(  x_{0}\right)  ,\ldots,\pi^{m}%
y_{l-1}+f_{l-1}\left(  x_{0}\right)  \right)  \left\vert dy\right\vert .
\]

We now note that $I(x_{0},m)=0$ unless $f_{i}\left(  x_{0}\right)  \equiv0$
$\operatorname{mod}$ $\pi^{m}$, $i=1,\ldots,l-1$; in this case, a simple
change of variables shows that
\[
I(x_{0},m)=q^{-m\left(  n-l+1\right)  }%
{\displaystyle\int\limits_{\left(  \pi^{m}R_{K}\right)  ^{l-1}}}
\delta\left(  z_{1},\ldots,z_{l-1}\right)  \left\vert dz\right\vert
=q^{-m\left(  n-l+1\right)  }.
\]
The Claim follows by observing that there are $N_{m}$ integrals of type
$I(x_{0},m)$ each of them equals $q^{-m\left(  n-l+1\right)  }$.
\end{proof}

\begin{theorem}
\label{theorem5} Assume that $V^{(l-1)}\left(  K\right)  $ is a smooth
$K$-algebraic variety of dimension $n-l+1$ with good reduction mod $P_{K}$,
and that $V^{(l)}\left(  K\right)  $ and $V^{(l-1)}\left(  K\right)  $ are
non-degenerate complete intersection varieties over $\overline{K}$. Fix a
rational simple polyhedral subdivision $\Sigma^{\ast}=\Sigma^{\ast}\left(
f_{1},\ldots,f_{l}\right)  $ of $\mathbb{R}_{+}^{n}$ subordinate to
$\Gamma\left(  f_{1},\ldots,f_{l}\right)  $. Then (1) $P\left(  t\right)  $ is
a rational function of $t$ with rational coefficients; (2) if $\gamma
_{\boldsymbol{f}}>-1$, then
\[
N_{m}\leq C\left(  K\right)  q^{m\left(  n-l+1+\gamma_{\boldsymbol{f}}\right)
}m^{n-l}.
\]

\end{theorem}

\begin{proof}
The first part follows from the rationality of $Z\left(  s,V^{(l-1)}%
,f_{l}\right)  $ (cf. Theorem \ref{theorem1} and Remark \ref{remark5}) by
\ (\ref{identity}).The \ second part follows by expanding in simple fractions
the left side of (\ref{identity}) and using Remark \ref{remark5A}.
\end{proof}

\subsection{Exponential Sums Along Smooth Algebraic Varieties}

\begin{lemma}
\label{lemma4}Let $f_{i}\left(  x\right)  \in R_{K}\left[  x_{1},\ldots
,x_{n}\right]  $, $f_{i}\left(  0\right)  =0$, $i=1,\ldots,l$, with $2\leq
l\leq n$. Assume that $V^{(l-1)}\left(  K\right)  $ is a smooth algebraic
variety of dimension $n-l+1$ with good reduction $\operatorname{mod}$ $P_{K}$.
If $z=u\pi^{-m}$, $m\in\mathbb{N}$, $u\in R_{K}^{\times}$, then
\begin{align*}
E(z)  &  :=\int\limits_{R_{K}^{n}}\delta\left(  f_{1}\left(  x\right)
,\ldots,f_{l-1}\left(  x\right)  \right)  \Psi\left(  zf_{l}(x)\right)  \mid
dx\mid\\
&  =\int\limits_{V^{(l-1)}(R_{K})}\Psi\left(  zf_{l}\left(  x\right)  \right)
\mid\gamma_{GL}\left(  x\right)  \mid\\
&  =q^{-m\left(  n-l+1\right)  }%
{\displaystyle\sum\limits_{y\in V^{(l-1)}\left(  R_{K}/P_{K}^{m}\right)  }}
\Psi\left(  zf_{l}(y)\right)  .
\end{align*}

\end{lemma}

\begin{proof}
The lemma follows from Remark \ref{osci_int} and the following identity:%
\[%
{\displaystyle\int\limits_{x_{0}+\left(  P_{K}^{m}\right)  ^{n}}}
\delta\left(  f_{1}\left(  x\right)  ,\ldots,f_{l-1}\left(  x\right)  \right)
\left\vert dx\right\vert =
\]%
\[
\left\{
\begin{array}
[c]{ll}%
q^{-m\left(  n-l+1\right)  }\text{,} & \text{if }x_{0}\in V^{(l-1)}\left(
R_{K}\right)  \text{ mod }P_{K}^{m}\\
& \\
0\text{,} & \text{otherwise.}%
\end{array}
\right.
\]
The proof of this identity is very close to the proof of Claim 1 in the proof
of Lemma \ref{lemma2}.
\end{proof}

\begin{theorem}
\label{theorem6}Let $f_{i}\left(  x\right)  \in R_{K}\left[  x_{1}%
,\ldots,x_{n}\right]  $, $f_{i}\left(  0\right)  =0$, for $i=1,\ldots,l$, with
$2\leq l\leq n$. Assume that $V^{(l-1)}\left(  K\right)  $ is a smooth
algebraic variety of dimension $n-l+1$ with good reduction mod $P_{K}$, and
that $V^{(l)}\left(  K\right)  $ and$\ V^{(l-1)}\left(  K\right)  $ are
non-degenerate complete intersection varieties. Fix a rational simple
polyhedral subdivision $\Sigma^{\ast}=\Sigma^{\ast}\left(  f_{1},\ldots
,f_{l}\right)  $ of $\mathbb{R}_{+}^{n}$ subordinate to $\Gamma\left(
\boldsymbol{f}\right)  $. Assume that the $l-$form $%
{\textstyle\bigwedge\nolimits_{i=1}^{l}}
df_{i}$ does not vanish on $V^{(l,\lambda)}\left(  K\right)  $, for any
$\lambda\in K^{\times}$. If $\gamma_{\boldsymbol{f}}>-1$, then
\[
\left\vert E(z)\right\vert \leq C\left(  K\right)  \left\vert z\right\vert
_{K}^{\gamma_{\boldsymbol{f}}}\left(  \log_{q}\left\vert z\right\vert
_{K}\right)  ^{n-l},
\]
for $\left\vert z\right\vert _{K}$ big enough, where \ is $C\left(  K\right)
$\ a positive constant.
\end{theorem}

\begin{proof}
The result follows from the previous lemma by applying Theorem \ref{theorem3}%
\ (2) and Remark \ref{remark10}.
\end{proof}

\section{\label{section6}Explicit Formulas for Local Zeta Functions Supported
on Non-degenerate Complete Intersection Varieties}

In this section $V^{(l)}\left(  K\right)  $ and $V^{(l-1)}\left(  K\right)  $
are convenient and non-degenerate complete intersection varieties over
$\overline{K}$, and $V^{(l-1)}\left(  K\right)  $ is not necessarily a
submanifold. We associate to $V^{(l-1)}$and $f_{l}$ the following local zeta function:%

\[
\mathcal{Z}\left(  s,V^{(l-1)},f_{l}\right)  :=\lim_{r\rightarrow+\infty}%
{\displaystyle\int\limits_{R_{K}^{n}}}
\delta_{r}\left(  f_{1}(x),\ldots,f_{l-1}(x)\right)  \left\vert f_{l}\left(
x\right)  \right\vert _{K}^{s}\mid dx\mid,
\]
where $s\in\mathbb{C}$, with $\operatorname{Re}(s)>0$, and $\mid dx\mid$ is
the normalized Haar measure of $K^{n}$.

We also define%

\[
\mathcal{Z}_{0}\left(  s,V^{(l-1)},f_{l}\right)  :=\lim_{r\rightarrow+\infty}%
{\displaystyle\int\limits_{\left(  P_{K}\right)  ^{n}}}
\delta_{r}\left(  f_{1}(x),\ldots,f_{l-1}(x)\right)  \left\vert f_{l}\left(
x\right)  \right\vert _{K}^{s}\mid dx\mid,
\]
for $\operatorname{Re}(s)>0$.

The following notation will be used in this section. Given a polynomial
mapping $\boldsymbol{f}=(f_{1},\ldots,f_{l}):K^{n}\rightarrow K^{l}$ over
$R_{K}$, and positive vector $a$, we set as before,%
\[
V_{a}^{(j)}\left(  \overline{K}\right)  =\left\{  \overline{z}\in\overline
{K}^{n}\mid\overline{f_{i,a}}\left(  \overline{z}\right)  =0,\text{
}i=1,\ldots,j\right\}  .
\]
In the case $a=0$, we take $V_{a}^{(j)}\left(  \overline{K}\right)
=V^{(j)}\left(  \overline{K}\right)  $. Let $\Delta$\ be a rational simplicial
cone spanned by $a_{i}$, $i=1,\ldots,e_{\Delta}$. We define the
\textit{barycenter} of $\Delta$\ as $b\left(  \Delta\right)  =%
{\textstyle\sum\nolimits_{i=1}^{e_{\Delta}}}
a_{i}$.

For $x=\left(  x_{1},\ldots,x_{n}\right)  \in R_{K}^{n}$, we define
$ord(x)=\left(  ord\left(  x_{1}\right)  ,\ldots,ord\left(  x_{n}\right)
\right)  \in\mathbb{N}^{n}$, and
\[
E_{\Delta}=\left\{  x\in R_{K}^{n}\mid ord(x)\in\Delta\right\}  .
\]

\begin{theorem}
\label{theorem7}Let $\boldsymbol{f}=(f_{1},\ldots,f_{l}):K^{n}\rightarrow
K^{l}$, $\boldsymbol{f}\left(  0\right)  =0$, $2\leq l\leq n$, be a convenient
polynomial mapping over $R_{K}$. Assume that $V^{(l)}\left(  K\right)  $ and
$V^{(l-1)}\left(  K\right)  $ are non-degenerate complete intersection
varieties over $\overline{K}$.

\noindent(1) There exists a real constant $c\left(  \Gamma_{1},\ldots
,\Gamma_{l}\right)  $ such that $\mathcal{Z}\left(  s,V^{(l-1)},f_{l}\right)
$ is holomorphic on $\operatorname{Re}(s)>c\left(  \Gamma_{1},\ldots
,\Gamma_{l}\right)  $.

\noindent(2) $\mathcal{Z}\left(  s,V^{(l-1)},f_{l}\right)  $ admits a
meromorphic continuation to the complex plane as a rational function of
$q^{-s}$, which can be computed from any given rational simplicial polyhedral
subdivision $\Sigma^{\ast}$\ of $\mathbb{R}_{+}^{n}$ subordinate to
$\Gamma\left(  \boldsymbol{f}\right)  $ as follows:
\[
\mathcal{Z}\left(  s,V^{(l-1)},f_{l}\right)  =L_{0}\left(  q^{-s}\right)  +%
{\displaystyle\sum\limits_{\Delta\in\Sigma^{\ast}}}
L_{\Delta}\left(  q^{-s}\right)  S_{\Delta}\left(  q^{-s}\right)  ,
\]
where $L_{0}\left(  q^{-s}\right)  $, $L_{\Delta}\left(  q^{-s}\right)  $, and
$S_{\Delta}\left(  q^{-s}\right)  $ are defined as follows:
\[%
\begin{array}
[c]{l}%
L_{0}\left(  q^{-s}\right)  =q^{-\left(  n-l+1\right)  }card\left(
V^{(l-1)}\left(  \overline{K}\right)  \cap\left\{  z\in\left(  \overline
{K}^{\times}\right)  ^{n}\mid\overline{f}_{l}\left(  z\right)  \neq0\right\}
\right) \\
\\
+\text{ }q^{-s-\left(  n-l+1\right)  }card\left(  V^{(l)}\left(  \overline
{K}\right)  \cap\left(  \overline{K}^{\times}\right)  ^{n}\right)  \left(
\frac{1-q^{-1}}{1-q^{-s-1}}\right)  ;
\end{array}
\]%
\[%
\begin{array}
[c]{l}%
L_{\Delta}\left(  q^{-s}\right)  =\\
\\
q^{-\left(  n-l+1\right)  }card\left(  V_{b(\Delta)}^{(l-1)}\left(
\overline{K}\right)  \cap\left\{  z\in\left(  \overline{K}^{\times}\right)
^{n}\mid\overline{f}_{l,b(\Delta)}\left(  z\right)  \neq0\right\}  \right) \\
\\
+q^{-s-\left(  n-l+1\right)  }\left(  \frac{1-q^{-1}}{1-q^{-s-1}}\right)
card\left(  V_{b(\Delta)}^{(l)}\left(  \overline{K}\right)  \cap\left(
\overline{K}^{\times}\right)  ^{n}\right)  ;
\end{array}
\]
and%
\[%
\begin{array}
[c]{l}%
S_{\Delta}\left(  q^{-s}\right)  =\\
\left(
{\textstyle\sum\nolimits_{h}}
q^{d\left(  h,\Gamma_{l}\right)  s+\sigma\left(  h\right)  -%
{\textstyle\sum\nolimits_{j=1}^{l-1}}
d\left(  h,\Gamma_{j}\right)  }\right)
{\displaystyle\prod\limits_{i=1}^{e_{\Delta}}}
\left(  \frac{q^{-d\left(  a_{i},\Gamma_{l}\right)  s-\sigma\left(
a_{i}\right)  +%
{\textstyle\sum\nolimits_{j=1}^{l-1}}
d\left(  a_{i},\Gamma_{j}\right)  }}{1-q^{-d\left(  a_{i},\Gamma_{l}\right)
s-\sigma\left(  a_{i}\right)  +%
{\textstyle\sum\nolimits_{j=1}^{l-1}}
d\left(  a_{i},\Gamma_{j}\right)  }}\right)  ,
\end{array}
\]
where $h$ runs through the elements of the set
\[
\mathbb{N}^{n}\cap\left\{
{\textstyle\sum\nolimits_{i=1}^{e_{\Delta}}}
\mu_{i}a_{i}\mid0\leq\mu_{i}<1\text{ for }i=1,\ldots,e_{\Delta}\right\}  .
\]
\noindent(3) If $V^{(l-1)}(K)$ is a submanifold of $K^{n}$, then
\[
\mathcal{Z}\left(  s,V^{(l-1)},f_{l}\right)  =Z\left(  s,V^{(l-1)}%
,f_{l}\right)  .
\]

\end{theorem}

\begin{proof}
We set for $r\in\mathbb{N}$,
\[
I_{0}^{(r)}\left(  q^{-s}\right)  :=%
{\displaystyle\int\limits_{\left(  R_{K}^{\times}\right)  ^{n}}}
\delta_{r}\left(  f_{1}\left(  x\right)  ,\ldots,f_{l-1}\left(  x\right)
\right)  \left\vert f_{l}\left(  x\right)  \right\vert _{K}^{s}\left\vert
dx\right\vert ,\text{ }\operatorname{Re}(s)>0,
\]
and
\[
I_{\Delta}^{(r)}\left(  q^{-s}\right)  :=%
{\displaystyle\int\limits_{E_{\Delta}}}
\delta_{r}\left(  f_{1}\left(  x\right)  ,\ldots,f_{l-1}\left(  x\right)
\right)  \left\vert f_{l}\left(  x\right)  \right\vert _{K}^{s}\left\vert
dx\right\vert ,\text{ }\operatorname{Re}(s)>0.
\]

Since $\mathbb{R}_{+}^{n}=\left\{  0\right\}
{\textstyle\bigcup}
\cup_{\Delta\in\Sigma^{\ast}}\Delta$, we have%
\[
\mathcal{Z}\left(  s,V^{(l-1)},f_{l}\right)  =\lim_{r\rightarrow+\infty}%
I_{0}^{(r)}\left(  q^{-s}\right)  +%
{\displaystyle\sum\limits_{\Delta\in\Sigma^{\ast}}}
\lim_{r\rightarrow+\infty}I_{\Delta}^{(r)}\left(  q^{-s}\right)  .
\]

The \ parts (1)-(2) of the theorem follows from the previous formula by using
the following two claims.

\begin{claim}
\label{claim1} $\lim_{r\rightarrow+\infty}I_{0}^{(r)}\left(  q^{-s}\right)  $
gives a holomorphic function for $\operatorname{Re}(s)>-1$. Furthermore,
$\ \lim_{r\rightarrow+\infty}I_{0}^{(r)}\left(  q^{-s}\right)  =L_{0}\left(
q^{-s}\right)  $.
\end{claim}

\begin{claim}
\label{claim2}Let $\Delta$ be a rational simplicial cone spanned by $a_{i}$,
$i=1,\ldots,e_{\Delta}$. Then $\lim_{r\rightarrow+\infty}I_{\Delta}%
^{(r)}\left(  q^{-s}\right)  $ gives a holomorphic function for
$\operatorname{Re}(s)>c\left(  \Delta\right)  $, where $c\left(
\Delta\right)  $\ is a real constant. In addition,
\[
\lim_{r\rightarrow+\infty}I_{\Delta}^{(r)}\left(  q^{-s}\right)  =L_{\Delta
}\left(  q^{-s}\right)  S_{\Delta}\left(  q^{-s}\right)  .
\]

\end{claim}

\textbf{Proof Claim }\ref{claim1}. Note that $I_{0}^{(r)}\left(
q^{-s}\right)  $ can be expressed as a finite sum of integrals of type
\[
I^{(r)}\left(  s,x_{0}\right)  :=%
{\displaystyle\int\limits_{x_{0}+\left(  P_{K}\right)  ^{n}}}
\delta\left(  f_{1}\left(  x\right)  ,\ldots,f_{l-1}\left(  x\right)  \right)
\left\vert f_{l}\left(  x\right)  \right\vert _{K}^{s}\left\vert dx\right\vert
,\text{ }\operatorname{Re}(s)>0,
\]
where $x_{0}$ runs through a fixed set of representatives of $\left(
R_{K}^{\times}\right)  ^{n}$ $\operatorname{mod}$ $P_{K}$. We now put
\begin{equation}
y_{i}:=\left\{
\begin{array}
[c]{lll}%
\frac{f_{i}\left(  x_{0}+\pi x\right)  -f_{i}\left(  x_{0}\right)  }{\pi} &  &
i=1,\ldots,l\\
&  & \\
x_{i} &  & i=l+1,\ldots,n,
\end{array}
\right.  \label{changevar}%
\end{equation}
and $y=\left(  y_{1},\ldots,y_{n}\right)  :=\rho\left(  x\right)  $. Since
$\boldsymbol{f}=\left(  f_{1},\ldots,f_{l}\right)  $ is strongly
non-degenerate over $\overline{K}$, $y=\rho\left(  x\right)  $ gives a measure
preserving $K$-bianalytic map from $R_{K}^{n}$ to itself (cf. \cite[Lemma
7.4.3]{I2}), and therefore%

\[
I^{(r)}\left(  s,x_{0}\right)  =q^{-n}%
{\displaystyle\int\limits_{R_{K}^{n}}}
\delta_{r}\left(  \pi y_{1}+f_{1}\left(  x_{0}\right)  ,\ldots,\pi
y_{l-1}+f_{l-1}\left(  x_{0}\right)  \right)  \left\vert \pi y_{l}%
+f_{l}\left(  x_{0}\right)  \right\vert _{K}^{s}\left\vert dy\right\vert .
\]
If $\overline{x_{0}}\notin V^{(l-1)}\left(  \overline{K}\right)  $, then
$\lim_{r\rightarrow+\infty}I^{(r)}\left(  s,x_{0}\right)  $ $=0$. Now, if
$\overline{x_{0}}\in V^{(l-1)}\left(  \overline{K}\right)  $ and
$\overline{x_{0}}\notin V^{(l)}\left(  \overline{K}\right)  $, we have%
\[
\lim_{r\rightarrow+\infty}I^{(r)}\left(  s,x_{0}\right)  =q^{-n}%
\lim_{r\rightarrow+\infty}%
{\displaystyle\int\limits_{R_{K}^{l-1}}}
\delta_{r}\left(  \pi y_{1}+f_{1}\left(  x_{0}\right)  ,\ldots,\pi
y_{l-1}+f_{l-1}\left(  x_{0}\right)  \right)  \left\vert dy\right\vert
\]%
\[
=q^{-\left(  n-l+1\right)  }\lim_{r\rightarrow+\infty}%
{\displaystyle\int\limits_{R_{K}^{l-1}}}
\delta_{r}\left(  z_{1},\ldots,z_{l-1}\right)  \left\vert dz\right\vert
=q^{-\left(  n-l+1\right)  }.
\]

If $\overline{x_{0}}\in V^{(l)}\left(  \overline{K}\right)  $, we have
\begin{align*}
\lim_{r\rightarrow+\infty}I^{(r)}\left(  s,x_{0}\right)   &  =q^{-\left(
n-l+1\right)  -s}\lim_{r\rightarrow+\infty}%
{\displaystyle\int\limits_{R_{K}^{l-1}}}
\delta_{r}\left(  z_{1},\ldots,z_{l-1}\right)  \left\vert dz\right\vert
{\displaystyle\int\limits_{R_{K}}}
\left\vert z_{l}\right\vert _{K}^{s}\left\vert dz_{l}\right\vert \\
&  =q^{-\left(  n-l+1\right)  -s}%
{\displaystyle\int\limits_{R_{K}}}
\left\vert z_{l}\right\vert _{K}^{s}\left\vert dz_{l}\right\vert ,\text{ for
}\operatorname{Re}(s)>0\text{.}%
\end{align*}
Therefore,
\[
(A)\text{ }\lim_{r\rightarrow+\infty}I^{(r)}\left(  s,x_{0}\right)  \text{
gives a holomorphic function for }\operatorname{Re}(s)>0\text{,}%
\]
and
\[
(B)\text{ }\lim_{r\rightarrow+\infty}I^{(r)}\left(  s,x_{0}\right)  =\left\{
\begin{array}
[c]{lll}%
0 & \text{if\qquad} & \overline{x_{0}}\notin V^{(l-1)}\left(  \overline
{K}\right) \\
&  & \\
q^{-\left(  n-l+1\right)  } & \text{if\qquad} & \left\{
\begin{array}
[c]{l}%
\overline{x_{0}}\in V^{(l-1)}\left(  \overline{K}\right)  \text{ }\\
\text{and}\\
\overline{x_{0}}\notin V^{(l)}\left(  \overline{K}\right)
\end{array}
\right. \\
&  & \\
q^{-s-\left(  n-l+1\right)  }\left(  \frac{1-q^{-1}}{1-q^{-s-1}}\right)  &
\text{if\qquad} & \overline{x_{0}}\in V^{(l)}\left(  \overline{K}\right)  .
\end{array}
\right.
\]

Now, the announced claim follows from (A) and (B).

\textbf{Proof of Claim } \ref{claim2}. We first note that%
\[
I_{\Delta}^{(r)}\left(  q^{-s}\right)  =%
{\displaystyle\sum\limits_{m\in\mathbb{N}^{n}\cap\Delta}}
\text{ \ }%
{\displaystyle\int\limits_{ord(x)=m}}
\delta_{r}\left(  f_{1}\left(  x\right)  ,\ldots,f_{l-1}\left(  x\right)
\right)  \left\vert f_{l}\left(  x\right)  \right\vert _{K}^{s}\left\vert
dx\right\vert .
\]
For $m=\left(  m_{1},\ldots,m_{n}\right)  \in\mathbb{N}^{n}\cap\Delta$ we
define%
\[
x_{i}=\pi^{m_{i}}u_{i},\text{ }u_{i}\in R_{K}^{\times}\text{, }i=1,\ldots
,n\text{.}%
\]
Then $\left\vert dx\right\vert =q^{-\sigma\left(  m\right)  }\left\vert
du\right\vert $,
\[
f_{i}\left(  x\right)  =\pi^{d\left(  m,\Gamma_{i}\right)  }\left(
f_{i,b\left(  \Delta\right)  }\left(  u\right)  +\pi g_{i}\left(  u\right)
\right)  =\pi^{d\left(  m,\Gamma_{i}\right)  }\widetilde{f}_{i}\left(
u\right)  \text{, }i=1,\ldots,l,
\]
and $I_{\Delta}^{(r)}\left(  q^{-s}\right)  $ equals
\[%
{\displaystyle\sum\limits_{m\in\mathbb{N}^{n}\cap\Delta}}
q^{-d\left(  m,\Gamma_{l}\right)  s-\sigma\left(  m\right)  }%
{\displaystyle\int\limits_{\left(  R_{K}^{\times}\right)  ^{n}}}
\delta_{r}\left(  \pi^{d\left(  m,\Gamma_{1}\right)  }\widetilde{f}_{1}\left(
u\right)  ,\ldots,\pi^{d\left(  m,\Gamma_{l-1}\right)  }\widetilde{f}%
_{l-1}\left(  u\right)  \right)  \left\vert \widetilde{f}_{l}\left(  u\right)
\right\vert _{K}^{s}\left\vert du\right\vert .
\]
Since
\begin{align*}
&
{\displaystyle\int\limits_{\left(  R_{K}^{\times}\right)  ^{n}}}
\delta_{r}\left(  \pi^{d\left(  m,\Gamma_{1}\right)  }\widetilde{f}_{1}\left(
u\right)  ,\ldots,\pi^{d\left(  m,\Gamma_{l-1}\right)  }\widetilde{f}%
_{l-1}\left(  u\right)  \right)  \left\vert \widetilde{f}_{l}\left(  u\right)
\right\vert _{K}^{s}\left\vert du\right\vert \\
&  =%
{\displaystyle\sum\limits_{x_{0}\in\left(  \mathcal{R}^{\times}\right)  ^{n}}}
{\displaystyle\int\limits_{x_{0}+\left(  P_{K}\right)  ^{n}}}
\delta_{r}\left(  \pi^{d\left(  m,\Gamma_{1}\right)  }\widetilde{f}_{1}\left(
u\right)  ,\ldots,\pi^{d\left(  m,\Gamma_{l-1}\right)  }\widetilde{f}%
_{l-1}\left(  u\right)  \right)  \left\vert \widetilde{f}_{l}\left(  u\right)
\right\vert _{K}^{s}\left\vert du\right\vert ,
\end{align*}
where $\mathcal{R}^{\times}$ denotes a fixed set of representatives of
$\overline{K}^{\times}$ in $R_{K}$, by changing variables as in
(\ref{changevar}), one gets that $\lim_{r\rightarrow+\infty}I_{\Delta}%
^{(r)}\left(  q^{-s}\right)  $ is equal to
\begin{equation}
L_{\Delta}\left(  q^{-s}\right)  \lim_{r\rightarrow+\infty}%
{\displaystyle\sum\limits_{m\in\mathbb{N}^{n}\cap\Delta}}
\left(  q^{-d\left(  m,\Gamma_{l}\right)  s-\left(  \sigma\left(  m\right)  -%
{\textstyle\sum\nolimits_{j=1}^{l-1}}
d\left(  m,\Gamma_{j}\right)  \right)  }%
{\displaystyle\int\limits_{B_{m}}}
\delta_{r}\left(  z_{1},\ldots,z_{l-1}\right)  \left\vert dz\right\vert
\right)  , \label{series}%
\end{equation}
where $B_{m}=\pi^{d\left(  m,\Gamma_{1}\right)  +1}R_{K}\times\ldots\times
\pi^{d\left(  m,\Gamma_{l-1}\right)  +1}R_{K}$.

For a strictly positive vertex $a$, we have $d\left(  a,\Gamma_{j}\right)
>0$. A vertex which is not strictly positive is a vector in the set $\left\{
E_{1},\ldots,E_{n}\right\}  $ because $\boldsymbol{f}$ is convenient. But for
such a vector $E_{i}$, we have $d\left(  E_{i},\Gamma_{j}\right)  =0$. For
strictly positive vertices $a_{i}$, we set%
\[
c(a_{i}):=-\frac{\sigma\left(  a_{i}\right)  -%
{\textstyle\sum\nolimits_{j=1}^{l-1}}
d\left(  a_{i},\Gamma_{j}\right)  }{d\left(  a_{i},\Gamma_{l}\right)  },
\]
and for a cone $\Delta$ spanned by $a_{i}$ with $i=1,\ldots,e_{\Delta}$, we
set
\[
c(\Delta):=\max_{i=1,\ldots,e_{\Delta}}\left\{  c(a_{i})\mid a_{i}\text{ is
strictly positive}\right\}  .
\]
By using the argument given by Denef and Hoornaert for the case $l=1$ (see
\cite{D-H}), one verifies that series%

\[%
{\displaystyle\sum\limits_{m\in\mathbb{N}^{n}\cap\Delta}}
q^{-d\left(  m,\Gamma_{l}\right)  s-\left(  \sigma\left(  m\right)  -%
{\textstyle\sum\nolimits_{j=1}^{l-1}}
d\left(  m,\Gamma_{j}\right)  \right)  }%
\]
converges absolutely on $\operatorname{Re}(s)>c(\Delta)$ and that it defines a
holomorphic function on $\operatorname{Re}(s)>c(\Delta)$, and furthermore,
\[
S_{\Delta}\left(  q^{-s}\right)  =%
{\displaystyle\sum\limits_{m\in\mathbb{N}^{n}\cap\Delta}}
q^{-d\left(  m,\Gamma_{l}\right)  s-\left(  \sigma\left(  m\right)  -%
{\textstyle\sum\nolimits_{j=1}^{l-1}}
d\left(  m,\Gamma_{j}\right)  \right)  }.
\]
By using the dominated convergence theorem, and the fact that series
$S_{\Delta}\left(  q^{-s}\right)  $ converges absolutely on $\operatorname{Re}%
(s)>c(\Delta)$, one gets that $\lim_{r\rightarrow+\infty}I_{\Delta}%
^{(r)}\left(  q^{-s}\right)  $ equals
\begin{align*}
&  L_{\Delta}\left(  q^{-s}\right)
{\displaystyle\sum\limits_{m\in\mathbb{N}^{n}\cap\Delta}}
\left(  q^{-d\left(  m,\Gamma_{l}\right)  s-\left(  \sigma\left(  m\right)  -%
{\textstyle\sum\nolimits_{j=1}^{l-1}}
d\left(  m,\Gamma_{j}\right)  \right)  }\text{ }\lim_{r\rightarrow+\infty
}\text{\ }%
{\displaystyle\int\limits_{B_{m}}}
\delta_{r}\left(  z_{1},\ldots,z_{l-1}\right)  \left\vert dz\right\vert
\right) \\
&  =L_{\Delta}\left(  q^{-s}\right)  S_{\Delta}\left(  q^{-s}\right)  ,
\end{align*}
since $\lim_{r\rightarrow+\infty}\delta_{r}=\delta$ on $S(K^{n})$.

Finally, we set $c\left(  \Gamma_{1},\ldots,\Gamma_{l}\right)  $ as
$\max\left\{  \cup_{\Delta\in\Sigma^{\ast}}c(\Delta)\cup\left\{  -1\right\}
\right\}  $, then $\mathcal{Z}\left(  s,V^{(l-1)},f_{l}\right)  $ is
holomorphic on $\operatorname{Re}(s)>c\left(  \Gamma_{1},\ldots,\Gamma
_{l}\right)  $.

If $V^{(l-1)}(K)$ is a submanifold of $K^{n}$, the previous reasoning shows
that the meromorphic continuation of $Z\left(  s,V^{(l-1)},f_{l}\right)
$\ can be computed from any given rational simplicial polyhedral subdivision
of $\mathbb{R}_{+}^{n}$ subordinate to $\Gamma\left(  \boldsymbol{f}\right)  $
using the explicit formula given in the statement of the theorem, then
$Z\left(  s,V^{(l-1)},f_{l}\right)  =\mathcal{Z}\left(  s,V^{(l-1)}%
,f_{l}\right)  $.
\end{proof}

\begin{remark}
\label{Remark12}If in Theorem \ref{theorem7}, we assume that $V^{(l)}\left(
K\right)  $ and $V^{(l-1)}\left(  K\right)  $ are non-degenerate complete
intersection varieties at the origin over $\overline{K}$, with same notation,
we have
\[
\mathcal{Z}_{0}\left(  s,V^{(l-1)},f_{l}\right)  =%
{\displaystyle\sum\limits_{\substack{\Delta\in\Sigma^{\ast}\\b(\Delta)\succ
0}}}
L_{\Delta}\left(  q^{-s}\right)  S_{\Delta}\left(  q^{-s}\right)  .
\]

\end{remark}

The following problems \ emerge naturally motivated by our previous theorem.

\begin{problem}
Let $V^{(l-1)}\left(  K\right)  $ be $K$-analytic subset, and let
$f_{l}:V^{(l-1)}\left(  K\right)  \rightarrow K$ be an $K$-analytic function.
For which pairs $\left(  f_{l},V^{(l-1)}\right)  $ is $\mathcal{Z}_{\Phi
}\left(  \omega,V^{(l-1)},f_{l}\right)  $ a \ rational function of $q^{-s}$?
\end{problem}

\begin{problem}
Let $V^{(l-1)}\left(  K\right)  $ be $K$-analytic \ submanifold of $U$, and
let $f_{l}:V^{(l-1)}\left(  K\right)  $ $\rightarrow K$ be an $K$-analytic
function. \ Is $\mathcal{Z}_{\Phi}\left(  \omega,V^{(l-1)},f_{l}\right)
=Z_{\Phi}\left(  \omega,V^{(l-1)},f_{l}\right)  ?$
\end{problem}

\section{Relative Monodromy and Poles of Local Zeta Functions}

With the obvious analogous definitions for strongly non-degeneracy over
$\mathbb{C}$ (respectively, strongly non-degeneracy at the origin), we have
\ the following. Suppose \ that $f_{1},\ldots,f_{l}$ are polynomials in $n$
variables with coefficients in a number field $F$ $\left(  \subseteq
\mathbb{C}\right)  $. Then we can consider $\boldsymbol{f}=(f_{1},\ldots
,f_{l})$ as a map $K^{n}\rightarrow K^{l}$ for any non-Archimedean completion
$K$ of $F$. If $\boldsymbol{f}$ is strongly non-degenerate over $\mathbb{C}$,
then $\boldsymbol{f}$ is strongly non-degenerate over $\overline{K}$\ for
almost all the completions \ $K$ of $F$. \ This \ fact follows by applying the
Weak Nullstellensatz.

We can associate to a complex non-degenerate polynomial mapping defined over a
number field $F$ a local zeta function, say $Z_{\Phi}(s,\chi,V^{(l-1)}%
,f_{l},K)$. In the case in which $V^{(l-1)}$ is an open subset of $K^{n}$,
there are several conjectures relating the real parts of the poles of local
zeta functions $Z_{\Phi}(s,\chi,V^{(l-1)},f_{l},K)$ and the eigenvalues of the
complex local monodromy (see \cite[and references therein]{D0}).

Consider two complete intersection varieties defined in a neighborhood $U$ of
the origin of $\mathbb{C}^{n}$:%
\begin{align*}
V^{\left(  l\right)  }\left(  \mathbb{C}\right)   &  =\left\{  z\in U\mid
f_{i}\left(  z\right)  =0\text{, }i=1,\ldots,l\right\}  ,\\
V^{\left(  l-1\right)  }\left(  \mathbb{C}\right)   &  =\left\{  z\in U\mid
f_{i}\left(  z\right)  =0\text{, }i=1,\ldots,l-1\right\}  .
\end{align*}
Assume that $V^{\left(  l\right)  }$, $V^{\left(  l-1\right)  }$ are germs of
non-degenerate complete intersection varieties at the origin and that
$V^{\left(  l-1\right)  }$ has at most an isolated singularity at the origin.
Consider the Milnor fibration
\[
f_{l}:E^{\ast}\left(  \varepsilon,\delta\right)  \rightarrow D_{\delta}^{\ast
},
\]
where%
\[
E^{\ast}\left(  \varepsilon,\delta\right)  =\left\{  z\in V^{\left(
l-1\right)  }\mid\left\Vert z\right\Vert <\epsilon\text{, }0<\left\vert
f_{l}\left(  z\right)  \right\vert \leq\delta\right\}  ,
\]
and
\[
D_{\delta}^{\ast}=\left\{  y\in\mathbb{C\mid}0<\left\vert y\right\vert
\leq\delta\right\}  .
\]
The zeta function of the monodromy of this fibration is called the
$l$\textit{-th principal zeta function of the Milnor fibration} of the mapping
$\boldsymbol{f}=(f_{1},\ldots,f_{l}):\left(  U,0\right)  \rightarrow\left(
\mathbb{C}^{l},0\right)  $, $2\leq l\leq n$,\ and it is denoted as $\zeta
_{l}\left(  t\right)  $ \cite{O2}. The corresponding monodromy is called
\textit{the }$l$\textit{-th principal monodromy of }$f_{l}$\textit{\ relative
to }$V_{l-1}$ \textit{at the origin}. In \cite{O1} M. Oka gave an explicit
formula for $\zeta_{l}\left(  t\right)  $ in terms of the Newton polyhedra
$\Gamma_{1},\ldots,\Gamma_{l}$. A similar result was also proved by A. N.
Kirillov in \cite{Kir}. These results are a generalization of Varchenko's
explicit formula for the monodromy zeta function of a non-degenerate analytic
function \cite{Var1}.

\begin{conjecture}
\label{conj}(Relative Monodromy Conjecture) For almost all the completions $K$
of $F$, if $s$ is a pole of $Z_{\Phi}(s,\chi,f_{l},V_{l-1},K)$, then
$\exp\left(  2\pi\sqrt{-1}\operatorname{Re}\left(  s\right)  \right)  $ is an
eigenvalue of the $l$\textit{-th principal monodromy of }$f_{l}$%
\textit{\ relative to }$V^{(l-1)}$ \textit{at the origin}.
\end{conjecture}

\section{Examples}

\subsection{\label{example1}Example}

We set $\boldsymbol{f}=\left(  f_{1},f_{2}\right)  \ $\ with $f_{2}\left(
x,y,z\right)  =x^{8}+y^{8}+z^{8}+x^{2}y^{2}z^{2}$ and $f_{1}\left(
x,y,z\right)  \allowbreak=x+y-z$. Then $\boldsymbol{f}$\ is strongly
non-degenerate over $\mathbb{C}$, and by our previous remarks, $\boldsymbol{f}%
$ is strongly non-degenerate over $\mathbb{F}_{p}$, for $p$ big enough. We
compute
\[
Z_{0}\left(  s,\boldsymbol{f}\right)  =Z_{0}\left(  s\right)  :=%
{\textstyle\int\limits_{\left(  p\mathbb{Z}_{p}\right)  ^{3}}}
\left\vert f_{2}\left(  x,y,z\right)  \right\vert _{p}^{s}\delta\left(
f_{1}\left(  x,y,z\right)  \right)  \left\vert dxdydz\right\vert .
\]
Note that $V\left(  K\right)  =\left\{  \left(  x,y,z\right)  \in
\mathbb{Q}_{p}^{3}\mid f_{1}\left(  x,y,z\right)  =f_{2}\left(  x,y,z\right)
=0\right\}  $ is a hyperplane section, by eliminating $z$, we get
$g(x,y):=x^{8}+y^{8}+\left(  x+y\right)  ^{8}+x^{2}y^{2}\left(  x+y\right)
^{2}$ which is a degenerate curve with respect to its Newton polyhedron.
Furthermore, by using the fact that $W\left(  K\right)  =\left\{  \left(
x,y,z\right)  \in\mathbb{Q}_{p}^{3}\mid f_{1}\left(  x,y,z\right)  =0\right\}
$ is a submanifold of $\mathbb{Q}_{p}^{3}$, we have
\[
Z_{0}\left(  s\right)  =%
{\textstyle\int\limits_{\left(  p\mathbb{Z}_{p}\right)  ^{2}}}
\left\vert g(x,y)\right\vert _{p}^{s}\left\vert dxdy\right\vert .
\]

In the calculation we use the dual diagram of $\Gamma\left(  \boldsymbol{f}%
\right)  $ (i.e., the set of normal vectors to the facets of $\Gamma\left(
\boldsymbol{f}\right)  $) and a simple polyhedral subdivision $\Sigma^{\ast
}\left(  \boldsymbol{f}\right)  =\Sigma^{\ast}=\left\{  \Delta_{i}\right\}
_{i}$ subordinate to $\Gamma\left(  \boldsymbol{f}\right)  $ is in figure
1.7.1 of \ \cite[p. 83]{O2}

The vertices (i.e., normal vectors to the facets of $\Gamma\left(
\boldsymbol{f}\right)  $) are as follows: $E_{1}=\left(  1,0,0\right)  $,
$E_{2}=\left(  0,1,0\right)  $, $E_{3}=\left(  0,0,1\right)  $, $P=\left(
1,1,1\right)  $, $P_{1}=\left(  2,1,1\right)  $, $P_{2}=\left(  1,2,1\right)
$, $P_{3}=\left(  1,1,2\right)  $.

For $A\subset\mathbb{R}_{+}^{3}$, we set $E_{A}=\left\{  \left(  x,y,z\right)
\in\mathbb{Z}_{p}^{3}\mid\left(  ord\left(  x\right)  ,ord\left(  y\right)
,ord\left(  z\right)  \right)  \in A\right\}  $ as before. Since we have a
disjoint union $\left(  \mathbb{R}_{>0}\right)  ^{3}=\cup_{\substack{\Delta
\in\Sigma^{\ast}\\b\left(  \Delta\right)  \succ0}}\Delta$, we have
\begin{align*}
Z_{0}\left(  s,\boldsymbol{f}\right)   &  =%
{\textstyle\sum\limits_{\substack{\Delta_{i}\in\Sigma^{\ast}\\b(\Delta
_{i})\succ0}}}
\text{ }%
{\textstyle\int\limits_{E_{\Delta_{i}}}}
\left\vert f_{2}\left(  x,y,z\right)  \right\vert _{p}^{s}\delta\left(
f_{1}\left(  x,y,z\right)  \right)  \left\vert dxdydz\right\vert \\
&  =:%
{\textstyle\sum\limits_{i=1}^{25}}
Z_{0,\Delta_{i}}\left(  s,\boldsymbol{f}\right)  ,
\end{align*}
where the cones are defined below.

\textbf{Claim 1}. \textit{If }$\Delta_{i}$\textit{\ has dimension }%
$3$\textit{, then }$Z_{0,\Delta_{i}}\left(  s,\boldsymbol{f}\right)  =0$.

We consider first the particular case of the cone $\Delta_{1}$ spanned by
$E_{1}$, $P_{1}$, $E_{3}$. In this case $E_{\Delta_{1}}$ equals
\[%
{\textstyle\bigcup\limits_{a,b,c\in\mathbb{N}\smallsetminus\left\{  0\right\}
}}
\left\{  \left(  x,y,z\right)  \in\mathbb{Z}_{p}^{3}\mid x=p^{a+2b}%
u,y=p^{b}v,z=p^{b+c}w\text{, }u,v,w\in\mathbb{Z}_{p}^{\times}\right\}  .
\]
Then $Z_{0,\Delta_{1}}\left(  s,\boldsymbol{f}\right)  $\ can be expressed as
\[%
{\textstyle\sum\limits_{a=1}^{\infty}}
\text{ }%
{\textstyle\sum\limits_{b=1}^{\infty}}
\text{ }%
{\textstyle\sum\limits_{c=1}^{\infty}}
\text{ }%
{\textstyle\int\limits_{p^{a+2b}\mathbb{Z}_{p}^{\times}\times p^{b}%
\mathbb{Z}_{p}^{\times}\times p^{b+c}\mathbb{Z}_{p}^{\times}}}
\left\vert f_{2}\left(  x,y,z\right)  \right\vert _{p}^{s}\delta\left(
f_{1}\left(  x,y,z\right)  \right)  \left\vert dxdydz\right\vert .
\]
By taking
\[
\left\{
\begin{array}
[c]{l}%
x=p^{a+2b}u\\
y=p^{b}v\\
z=p^{b+c}w
\end{array}
\right.  ,\text{ }%
\begin{array}
[c]{cc}%
\left\vert dxdydz\right\vert =p^{a-4b-c} & \left\vert dudvdw\right\vert
\end{array}
,
\]
with $u,v,w\in\mathbb{Z}_{p}^{\times}$ as a change of variables in the
previous integral, $Z_{\Delta_{1}}\left(  s,\boldsymbol{f}\right)  $ becomes%
\[%
{\textstyle\sum\limits_{a=1}^{\infty}}
\text{ }%
{\textstyle\sum\limits_{b=1}^{\infty}}
\text{ }%
{\textstyle\sum\limits_{c=1}^{\infty}}
\text{ }p^{-a-4b-c-8bs}%
{\textstyle\int\limits_{\left(  \mathbb{Z}_{p}^{\times}\right)  ^{3}}}
\left\vert \widetilde{f}_{2}\left(  u,v,w\right)  \right\vert _{p}^{s}%
\delta\left(  p^{b}\widetilde{f}_{1}\left(  u,v,w\right)  \right)  \left\vert
dudvdw\right\vert ,
\]
where $\widetilde{f}_{1}\left(  u,v,w\right)  =v+p^{a+b}u-p^{c}w$,
$\widetilde{f}_{2}\left(  u,v,w\right)  =v^{8}+p^{8a+8b}u^{8}+p^{8c}%
w^{8}+p^{2a+2c}u^{2}v^{2}w^{2}$. Since
\[
\delta\left(  p^{b}\widetilde{f}_{1}\left(  u,v,w\right)  \right)  =0\text{
for every }\left(  u,v,w\right)  \in\left(  \mathbb{Z}_{p}^{\times}\right)
^{3}\text{,}%
\]
we conclude that $Z_{\Delta_{1}}\left(  s,\boldsymbol{f}\right)  =0$.

Let $\Delta_{i}$ be an arbitrary cone of dimension three. For any $a\in
\Delta_{i}$, $F\left(  a,\Gamma_{1}\right)  $ consists of a monomial, say $x$,
then $\widetilde{f}_{1}\left(  u,v,w\right)  =u+p^{c(a)}u-p^{c^{\prime}(a)}w$,
with $c(a)$, $c^{\prime}(a)\in\mathbb{N\smallsetminus}\left\{  0\right\}  $.
Note that the only monomials we obtain are $x$, $y$, and $z$. We now use the
previous argument to conclude that $Z_{0,\Delta_{i}}\left(  s,\boldsymbol{f}%
\right)  =0$.

\textbf{Claim 2}.\textit{\ If }$\Delta_{i}$\textit{\ is a cone of dimension
two, then the possible values of }$Z_{0,\Delta_{i}}\left(  s,\boldsymbol{f}%
\right)  $\textit{\ are as follow:}

\begin{center}%
\begin{tabular}
[c]{|l|l|l|}\hline
$\text{Cones}$ & $\text{Generators}$ & $Z_{0,\Delta_{i}}\left(
s,\boldsymbol{f}\right)  $\\\hline
$\Delta_{10}$ & $E_{3},P_{1}$ & $0$\\\hline
$\Delta_{11}$ & $E_{1},P_{1}$ & $p^{-1}\left(  1-p^{-1}\right)  \left(
\frac{p^{-3-8s}}{1-p^{-3-8s}}\right)  $\\\hline
$\Delta_{12}$ & $P_{1},P_{3}$ & $0$\\\hline
$\Delta_{13}$ & $P_{3},E_{3}$ & $p^{-1}\left(  1-p^{-1}\right)  \left(
\frac{p^{-3-8s}}{1-p^{-3-8s}}\right)  $\\\hline
$\Delta_{14}$ & $P_{3},E_{2}$ & $0$\\\hline
$\Delta_{15}$ & $P_{3},P_{2}$ & $0$\\\hline
$\Delta_{16}$ & $P_{3},P$ & $\left(  1-p^{-1}\right)  ^{2}\left(
\frac{p^{-3-8s}}{1-p^{-3-8s}}\right)  \left(  \frac{p^{-2-6s}}{1-p^{-2-6s}%
}\right)  $\\\hline
$\Delta_{17}$ & $P,P_{1}$ & $\left(  1-p^{-1}\right)  ^{2}\left(
\frac{p^{-3-8s}}{1-p^{-3-8s}}\right)  \left(  \frac{p^{-2-6s}}{1-p^{-2-6s}%
}\right)  $\\\hline
$\Delta_{18}$ & $P_{1},P_{2}$ & $0$\\\hline
$\Delta_{19}$ & $P,P_{2}$ & $\left(  1-p^{-1}\right)  ^{2}\left(
\frac{p^{-3-8s}}{1-p^{-3-8s}}\right)  \left(  \frac{p^{-2-6s}}{1-p^{-2-6s}%
}\right)  $\\\hline
$\Delta_{20}$ & $P_{2},E_{1}$ & $0$\\\hline
$\Delta_{21}$ & $P_{2},E_{2}$ & $p^{-1}\left(  1-p^{-1}\right)  \left(
\frac{p^{-3-8s}}{1-p^{-3-8s}}\right)  $\\\hline
\end{tabular}

\end{center}

By using the same reasoning \ as in the calculation of $Z_{0,\Delta_{1}%
}\left(  s,\boldsymbol{f}\right)  $, we have $Z_{0,\Delta_{i}}\left(
s,\boldsymbol{f}\right)  =0$, for $i=10,12,14,18,20$. We now compute
$Z_{0,\Delta_{11}}\left(  s,\boldsymbol{f}\right)  $. By using the same
reasoning and notation as in the calculation of $Z_{0,\Delta_{1}}\left(
s,\boldsymbol{f}\right)  $, one gets the following expansion for
$Z_{0,\Delta_{11}}\left(  s,\boldsymbol{f}\right)  $:%
\[%
{\textstyle\sum\limits_{a=1}^{\infty}}
\text{ }%
{\textstyle\sum\limits_{b=1}^{\infty}}
\text{ }p^{-a-4b-8bs}%
{\textstyle\int\limits_{\left(  \mathbb{Z}_{p}^{\times}\right)  ^{3}}}
\left\vert \widetilde{f}_{2}\left(  u,v,w\right)  \right\vert _{p}^{s}%
\delta\left(  p^{b}\widetilde{f}_{1}\left(  u,v,w\right)  \right)  \left\vert
dudvdw\right\vert ,
\]
where $\widetilde{f}_{1}\left(  u,v,w\right)  =v-w+p^{a+b}u$, $\widetilde
{f}_{2}\left(  u,v,w\right)  =v^{8}+w^{8}+p^{8a+8b}u^{8}+p^{2a}u^{2}v^{2}%
w^{2}$. Since
\[
\delta\left(  p^{b}\widetilde{f}_{1}\left(  u,v,w\right)  \right)  =0\text{
for }\left(  u,v,w\right)  \notin\left\{  \left(  u,v,w\right)  \in\left(
\mathbb{Z}_{p}^{\times}\right)  ^{3}\mid w=v+p^{a+b}u\right\}  :=Y,
\]
and $\left.  \left\vert \widetilde{f}_{2}\left(  u,v,w\right)  \right\vert
_{p}^{s}\right\vert _{Y}=1$, and assuming $p>2$, \ one gets%
\[
Z_{0,\Delta_{11}}\left(  s,\boldsymbol{f}\right)  =%
{\textstyle\sum\limits_{a=1}^{\infty}}
\text{ }%
{\textstyle\sum\limits_{b=1}^{\infty}}
\text{ }p^{-a-4b-8bs}%
{\textstyle\int\limits_{\left(  \mathbb{Z}_{p}^{\times}\right)  ^{3}}}
\delta\left(  p^{b}\widetilde{f}_{1}\left(  u,v,w\right)  \right)  \left\vert
dudvdw\right\vert .
\]
Fix a set of representatives $\mathcal{R}$ of $\ \left(  \mathbb{F}%
_{p}^{\times}\right)  ^{3}$ in $\left(  \mathbb{Z}_{p}^{\times}\right)  ^{3}$.
Then
\[%
{\textstyle\int\limits_{\left(  \mathbb{Z}_{p}^{\times}\right)  ^{3}}}
\delta\left(  p^{b}\widetilde{f}_{1}\left(  u,v,w\right)  \right)  \left\vert
dudvdw\right\vert
\]
is a finite sum of integrals of type
\begin{align*}
&
{\textstyle\int\nolimits_{\xi+p\mathbb{Z}_{p}^{3}}}
\delta\left(  p^{b}\widetilde{f}_{1}\left(  u,v,w\right)  \right)  \left\vert
dudvdw\right\vert \\
&  =p^{-3}%
{\textstyle\int\nolimits_{\mathbb{Z}_{p}^{3}}}
\delta\left(  p^{b}\widetilde{f}_{1}\left(  \xi_{1}+pu,\xi_{2}+pv,\xi
_{3}+pw\right)  \right)  \left\vert dudvdw\right\vert ,
\end{align*}
where $\xi=\left(  \xi_{1},\xi_{2},\xi_{3}\right)  \in\mathcal{R}$, and
$\overline{\xi}\in\left(  \mathbb{F}_{p}^{\times}\right)  ^{3}$. By taking%
\[
\left\{
\begin{array}
[c]{l}%
z_{1}=\frac{\widetilde{f}_{1}\left(  \xi_{1}+pu,\xi_{2}+pv,\xi_{3}+pw\right)
-\widetilde{f}_{1}\left(  \xi_{1},\xi_{2},\xi_{3}\right)  }{p}\\
z_{2}=u\\
z_{3}=w
\end{array}
\right.  ,
\]
as a change of variables in the previous integral, and using \ the implicit
function theorem (see e.g. \cite[Lemma 7.4.3]{I2}) one gets%
\[
p^{-3}%
{\textstyle\int\nolimits_{\mathbb{Z}_{p}}}
\delta\left(  p^{b}\left(  pz_{1}+\widetilde{f}_{1}\left(  \xi_{1},\xi_{2}%
,\xi_{3}\right)  \right)  \right)  \left\vert dz_{1}\right\vert .
\]
If $\overline{\widetilde{f}_{1}\left(  \xi_{1},\xi_{2},\xi_{3}\right)  }\neq
0$, then the integral is zero. Now if $\overline{\widetilde{f}_{1}\left(
\xi_{1},\xi_{2},\xi_{3}\right)  }=0$ (note that this equation has $\left(
p-1\right)  ^{2}$\ solutions in $\left(  \mathbb{F}_{p}^{\times}\right)  ^{3}%
$), the integral becomes
\[
p^{-3}%
{\textstyle\int\nolimits_{\mathbb{Z}_{p}}}
\delta\left(  p^{b+1}z_{1}\right)  \left\vert dz_{1}\right\vert =p^{-2+b}.
\]
Therefore
\begin{align*}
Z_{0,\Delta_{11}}\left(  s,\boldsymbol{f}\right)   &  =p^{-2}\left(
p-1\right)  ^{2}%
{\textstyle\sum\limits_{b=1}^{\infty}}
\text{ }p^{-a-3b-8bs}\\
&  =p^{-1}\left(  1-p^{-1}\right)  \frac{p^{-3-8s}}{1-p^{-3-8s}}.
\end{align*}
The other integrals are calculated in a similar form, as well as the following ones.

\textbf{Claim 3}.\textit{\ If }$\Delta_{i}$\textit{\ is a cone of dimension
one, then the possible values of }$Z_{0,\Delta_{i}}\left(  s,\boldsymbol{f}%
\right)  $\textit{\ are as follow:}

\begin{center}%
\begin{tabular}
[c]{|l|l|l|}\hline
$\text{Cones}$ & $\text{Generators}$ & $Z_{0,\Delta_{i}}\left(
s,\boldsymbol{f}\right)  $\\\hline
$\Delta_{22}$ & $P$ & $\left(  1-p^{-1}\right)  ^{2}\left(  \frac{p^{-2-2s}%
}{1-p^{-2-2s}}\right)  $\\\hline
$\Delta_{23}$ & $P_{1}$ & $%
\begin{array}
[c]{c}%
\left(  \left(  p-1\right)  ^{3}-N\right)  p^{-2}\left(  \frac{p^{-3-8s}%
}{1-p^{-3-8s}}\right)  +\\
Np^{-s+2}\left(  \frac{1-p^{-1}}{1-p^{-s-1}}\right)  \left(  \frac{p^{-3-8s}%
}{1-p^{-3-8s}}\right)
\end{array}
$\\\hline
$\Delta_{24}$ & $P_{2}$ & $%
\begin{array}
[c]{c}%
\left(  \left(  p-1\right)  ^{3}-N\right)  p^{-2}\left(  \frac{p^{-3-8s}%
}{1-p^{-3-8s}}\right)  +\\
Np^{-s+2}\left(  \frac{1-p^{-1}}{1-p^{-s-1}}\right)  \left(  \frac{p^{-3-8s}%
}{1-p^{-3-8s}}\right)
\end{array}
$\\\hline
$\Delta_{25}$ & $P_{3}$ & $%
\begin{array}
[c]{c}%
\left(  \left(  p-1\right)  ^{3}-N\right)  p^{-2}\left(  \frac{p^{-3-8s}%
}{1-p^{-3-8s}}\right)  +\\
Np^{-s+2}\left(  \frac{1-p^{-1}}{1-p^{-s-1}}\right)  \left(  \frac{p^{-3-8s}%
}{1-p^{-3-8s}}\right)
\end{array}
$\\\hline
\end{tabular}

\end{center}

where $N=\left(  p-1\right)  ^{2}\left\{  1+\left(  \frac{-2}{p}\right)
\right\}  $, $p>2$, and $\left(  \frac{\cdot}{p}\right)  $ is the Legendre symbol.

The real parts of the poles of $Z_{0}\left(  s,\boldsymbol{f}\right)  $ belong
to the set $\left\{  -1,\frac{-3}{8},\frac{-1}{3}\right\}  $; $\beta
_{\boldsymbol{f}}=\frac{-1}{3}$, and the multiplicity of the corresponding
pole is one.

On the other hand, $\varsigma_{2}\left(  t\right)  =\left(  1-t^{6}\right)
\left(  1-t^{8}\right)  ^{3}$ (see \cite[Example (I-2), pg. 220]{O2}).
Therefore the relative monodromy conjecture holds in this example.

\subsection{\label{Example2}Example}

We set $\boldsymbol{f}=\left(  f_{1},f_{2}\right)  $ with $f_{1}%
(x,y)=x^{n}+y^{n}$, $f_{2}(x,y)=x^{4}+y^{4}+xy$. Then $\boldsymbol{f}$\ is
strongly non-degenerate at the origin over $\mathbb{C}$, and by our previous
remarks, $\boldsymbol{f}$\ is strongly non-degenerate at the origin over
$\mathbb{F}_{p}$, for $p$ big enough. We compute $\mathcal{Z}_{0}%
(s,V^{(1)},f_{2})$, by using Theorem \ref{theorem7} and Remark \ref{Remark12}.
We use the simplicial polyhedral subdivision $\Sigma^{\ast}\left(
\boldsymbol{f}\right)  $ subordinate to $\Gamma\left(  \boldsymbol{f}\right)
$ showed in figure 2. \ The vertices are as follows: $E_{1}=(0,1)$,
$P_{1}=(1,3)$, $P_{2}=(1,1)$, $P_{3}=(3,1)$, $E_{2}=(1,0)$. 

Only the cone generated by $P_{2}$ gives a non-zero contribution to the local
zeta function,%

\[
\mathcal{Z}_{0}\left(  s,V^{(1)},f_{2}\right)  =\frac{card\left(  \left\{
\left(  x,y\right)  \in\left(  \mathbb{F}_{p}^{\times}\right)  ^{2}\mid
x^{n}+y^{n}=0\right\}  \right)  p^{-2s-3+n}}{1-p^{-2s-2+n}}.
\]
Note that the real parts of the poles of $\mathcal{Z}_{0}\left(
s,V^{(1)},f_{2}\right)  $ are zero or positive depending if $n=2$ or if $n>2$.

\end{document}